%% file: main.tex
\documentclass[preprint]{imsart}

\usepackage{graphicx}
\usepackage{amssymb, amsmath}
\usepackage{mathrsfs}
\usepackage{amsthm}
\usepackage{mathtools}
\usepackage{enumerate}
\usepackage[shortlabels]{enumitem}

\RequirePackage[OT1]{fontenc}
\RequirePackage[colorlinks,citecolor=blue,urlcolor=blue]{hyperref}

\arxiv{arXiv:1912.06635}

\startlocaldefs

\usepackage{xcolor}

\newcommand{\R}{\mathbb R}
\renewcommand{\P}{\mathbb P}

\newcommand{\arcsinh}{\operatorname{arcsinh}}
\newcommand{\unif}{\mathrm{Unif}}
\newcommand{\Leb}{\mathrm{Leb}}

\newcommand{\norm}[1]{\| #1 \| }

\newtheorem{theorem}{Theorem}[section]
\newtheorem{lemma}[theorem]{Lemma}
\newtheorem{proposition}[theorem]{Proposition}

\def\calB{\mathcal B}
\def\calC{\mathcal C}
\def\calD{\mathcal D}

\def\calI{\mathcal I}
\def\calJ{\mathcal J}

\def\calL{\mathcal L}
\def\calM{\mathcal M}

\def\calO{\mathcal O}
\def\calP{\mathcal P}
\def\calQ{\mathcal Q}

\def\bE{\mathbb E}

\def\bN{\mathbb N}

\def\bP{\mathbb P}

\def\bR{\mathbb R}

\def\bT{\mathbb T}

\def\b{\beta}

\endlocaldefs

\begin{document}

\begin{frontmatter}

\title{Large deviations for the empirical measure of the zig-zag process}
\runtitle{Empirical measure of the zig-zag process}


\author{\fnms{Joris} \snm{Bierkens}\ead[label=e1]{j.bierkens@vu.nl}}
\address{Department of Mathematics\\
Vrije Universiteit \\
De Boelelaan 1081a \\
1081 HV Amsterdam\\ 
The Netherlands \\
\printead{e1}}
\author{\fnms{Pierre} \snm{Nyquist}\ead[label=e2]{pierren@kth.se}}
\address{Department of Mathematics \\
KTH \\
100 44\\
Stockholm, Sweden\\
\printead{e2}}
\and
\author{\fnms{Mikola C.} \snm{Schlottke}\ead[label=e3]{M.Schlottke@tue.nl}}
\address{Department of Mathematics and Computer Science\\
Technische Universiteit Eindhoven \\
PO Box 513 \\
5600 MB Eindhoven \\
The Netherlands \\
\printead{e3}}

\runauthor{J.\ Bierkens, P.\ Nyquist, M.\ C.\ Schlottke}

\begin{abstract}
The zig-zag process is a piecewise deterministic Markov process in position and velocity space. The process can be designed to have an arbitrary Gibbs type marginal probability density for its position coordinate, which makes it suitable for Monte Carlo simulation of continuous probability distributions. An important question in assessing the efficiency of this method is how fast the empirical measure converges to the stationary distribution of the process. In this paper we provide a partial answer to this question by characterizing the large deviations of the empirical measure from the stationary distribution. Based on the Feng-Kurtz approach, we develop an abstract framework aimed at encompassing piecewise deterministic Markov processes in position-velocity space. We derive explicit conditions for the zig-zag process to allow the Donsker-Varadhan variational formulation of the rate function, both for a compact setting (the torus) and one-dimensional Euclidean space. Finally we derive an explicit expression for the Donsker-Varadhan functional for the case of a compact state space and use this form of the rate function to address a key question concerning the optimal choice of the switching rate of the zig-zag process.
\end{abstract}

\begin{keyword}[class=MSC]
\kwd[Primary ]{60F10}
\kwd[; secondary ]{60J25}
\end{keyword}

\begin{keyword}
\kwd{Large deviations} 
\kwd{empirical measure}
\kwd{piecewise deterministic Markov process}
\kwd{zig-zag process}
\end{keyword}

\end{frontmatter}

\section{Introduction}
The problem of sampling from a given high-dimensional probability distribution arises in a wide range of applications, for example chemistry, physics, statistics and engineering. Specifically, consider the task of sampling from a distribution $\pi$ with density with respect to Lebesgue measure given by
\begin{align*}
    \pi (y) \propto \textrm{exp}\{ - U(y) \}, \ y \in E,     
\end{align*}
for some potential function $U: E \to \bR$ and state space $E$. The most common approach is to use Markov chain Monte Carlo (MCMC) methods, which are now essential tools in areas such as computational statistics, molecular dynamics and machine learning, see e.g. \cite{RobertCasella2004,AsmussenGlynn07,AndrieuEtAl03} and references therein.

The idea behind MCMC is to construct a Markov process $Y_t$ with $\pi$ as the invariant measure and use the corresponding empirical measure to obtain approximations. For example, under ergodicity, for any observable $f \in L^1 (\pi)$ and $t>0$ large, $\frac{1}{t} \int _0 ^t f(Y_s) ds$ can be used to approximate the expected value $\int_E f (y) \pi (dy)$. 

Although many standard MCMC constructions, such as the Metropolis-Hastings algorithm \cite{Metropolis1953}, can be used to sample from essentially any target distribution $\pi$, most suffer from slow convergence to the invariant distribution or heavy computational costs per iteration. Designing new, efficient dynamics has therefore become an important research direction within applied probability.

Over the last decade piecewise deterministic Markov processes (PDMPs) \cite{davis1984piecewise} have emerged as a new tool for the numerical simulation of probability distributions, potentially mitigating both problems. The two main examples of such processes are the Bouncy Particle Sampler and the Zig-Zag Sampler \cite{BouchardCoteVollmerDoucet2017,bierkens2019zig}, after similar ideas appeared first in \cite{PetersDeWith2012} and  \cite{Monmarche2016}. The idea of using PDMPs extends the ubiquitous discrete time MCMC methodology towards a new continuous time approach, having several advantageous aspects. First, by construction PDMPs are non-reversible Markov processes, which typically results in a smaller asymptotic variance as compared to reversible methods (see e.g. \cite{Andrieu2019a} for a recent study). Furthermore, computation of PDMP trajectories on a computer may be carried out using a subsampling technique, reducing the computational effort required in each iteration; we refer to e.g. \cite{bierkens2019zig,Fearnhead2016a} for details of the computational aspects.

In order to employ this new PDMP methodology a solid understanding of mathematical properties of these methods is necessary. Whereas the theoretical properties of PDMPs have been an active research area in recent years, our understanding of the performance of the corresponding MCMC methods is still incomplete. In particular knowledge of the speed of convergence of time averages is essential in choosing the most suitable sampling technology for a particular problem and in tuning the parameters of the chosen method. In the spirit of recent work on empirical measure large deviations in the MCMC context \cite{dupuis2012infinite, rey2015irreversible}, we propose the use of large deviation results for studying and comparing the performance of PDMPs. 

In summary, the main contributions of this work are as follows:
\begin{itemize}
    \item A semigroup approach to establish the large deviations principle for the empirical measures of a class of Markov processes satisfying assumptions aimed at position-velocity PDMPs.
    \item The large deviations principle for empirical measures of the zig-zag process in both a compact and non-compact setting.
    \item Derivation of an explicit form of the rate function associated with the zig-zag process.
    \item Evaluation of the zig-zag rate function as a function of the additional switching rate $\gamma$, providing an answer to a key question about the switching rate.
\end{itemize}


The study of large deviations for empirical measures dates back to the work by Donsker and Varadhan \cite{donsker1975asymptoticI, donsker1975asymptoticII, DonskerVaradhan76}. In the simulation context it is well-known that for rare-event simulation, sample-path large deviations play an important r\^ole in evaluating and designing efficient algorithms ; see \cite{AsmussenGlynn07, Bucklew04, BudhirajaDupuis2019} and references therein. 
In contrast, empirical measure large deviations are much less explored as a tool for analysing Monte Carlo methods. Standard measures for analysing the efficiency of methods based on ergodic Markov processes include the spectral gap of the associated semigroup and the asymptotic variance for given observables, see for example \cite{Rosenthal2003, BedardRosenthal08,DiaconisHolmesNeal00,FrankeEtAl10,FrigessiEtAl93,HwangHwangSheu05, MengersenTweedie96, RobertsRosenthal04}. 
However these measures are not necessarily appropriate for studying the rate of convergence, as they only link indirectly to the empirical measure, the quantity of interest in Monte Carlo methods. Empirical measure large deviations on the other hand connect explicitly to the relevant properties, such as the transient behaviour of the underlying process. In a similar spirit, \cite{BirrellRB19} recently used concentration inequalities to obtain non-asymptotic performance guarantees for PDMPs.

To the best of our knowledge, the first results using empirical measure large deviations for analysing MCMC methods were \cite{plattner2011infinite, dupuis2012infinite}. Therein empirical measure large deviations, specifically the associated rate function, was proposed as a tool for analysing parallel tempering, one of the computational workhorses of the physical sciences, leading to a new type of simulation method (infinite swapping). In the subsequent work \cite{DDN2018} empirical measure large deviations were again used, combined with associated stochastic control problems, to analyse the convergence properties of these algorithms. Similarly, in \cite{rey2015irreversible} Rey-Bellet and Spiliopoulos use empirical measure large deviations to analyse the performance of certain non-reversible MCMC samplers.

The work by Donsker and Varadhan is the starting point for many results and application of empirical measure large deviations and their work has been extended in numerous directions, see e.g.\ \cite{DemboZeitouni1994, FengKurtz06, BudhirajaDupuis2019} for an overview and further references. However, naively applying the existing theory to PDMPs does not work since the transition probabilities are not sufficiently regular: for every $t > 0$ there is a positive probability that the process has not switched by time $t$, resulting in an atomic component of the Markov transition kernel. As a first step towards using empirical measure large deviations for analysing the performance of PDMPs we must therefore establish the relevant large deviations principles. 

The focus of this paper is to establish general large deviation results aimed at PDMPs and then specialize to the zig-zag process. In the process of proving the necessary large deviation results we consider a general class of Markov processes that can have position-velocity PDMPs, such as the bouncy particle and the zig-zag samplers, as special cases. In particular, this class includes processes that are not of diffusion type and non-reversible processes. When specialising to the zig-zag process, we derive an explicit form of the rate function, going beyond the variational form typical for results of Donsker-Varadhan-type. To the best of our knowledge this is the first instance where an explicit form of the rate function has been obtained for non-reversible processes that do not have a drift-diffusion character. 

A key question for using the zig-zag process for MCMC is whether or not it is advantageous for convergence to use the minimal (canonical) switching rates, or if one should allow for additional switches according to a fixed refreshment rate $\gamma >0$. Our analysis of the rate function associated with the zig-zag process allows us to give a partial answer to this question: in Section \ref{sec:rate} we establish that the rate function is decreasing as a function of the additional rate $\gamma$, establishing that from a large deviations perspective it is optimal to use the smallest possible rates, i.e.\ set $\gamma =0$. This goes in the opposite direction of the conclusion drawn from a spectral analysis (see \cite[Section 7.3]{BierkensVerduynLunel2019}), which shows at least a small benefit of increasing gamma beyond zero. This highlights the different nature of convergence of empirical averages (by studying large deviations or asymptotic variance, see e.g. \cite{Andrieu2019a,BierkensDuncan2016}) and convergence to equilibrium, using e.g. the spectral gap to describe rate of convergence; see \cite{Rosenthal2003} for more on this phenomenon. Our conclusion is in line with the earlier observation that having more non-reversibility increases the rate function \cite{rey2015irreversible}: one can view increasing $\gamma$ as decreasing the extent of non-reversibility inherent to the process.

Evaluation of the large deviation rate function for empirical measures, beyond the variational form given by Donsker and Varadhan, is typically a challenging task. For the diffusion setting, including both reversible and non-reversible processes, see \cite{dupuis2018large} and the references therein. In \cite{dupuis2015large} the authors consider reversible jump Markov processes and use stochatic control and weak convergence arguments to derive an explicit form of the rate function. Lastly, in the MCMC context, \cite{rey2015irreversible} consider diffusion processes on a compact manifold where the drift can be decomposed into sufficiently smooth reversible and non-reversible parts. The rate function can then be expressed in terms of the rate function of a related reversible diffusion and the solution of an elliptic PDE associated with the non-reversible component of the drift.

The proofs of the large deviation results in this paper are based on the general Hamilton-Jacobi approach to empirical measures developed by Feng and Kurtz in \cite[Chapter 12]{FengKurtz06}. We describe this approach, in the context of this paper, in more detail in Section \ref{sec:aux}.


The remainder of the paper is organised as follows. In Section \ref{sec:prelim} we give the necessary preliminaries: notation and relevant definitions, background on the zig-zag process and empirical measure large deviations. In particular we recall well-known large deviation results for empirical measures by Donsker and Varadhan. The main results are then presented in Section \ref{sec:main}. The section is split into the main assumptions and general large deviation statements (Section \ref{sec:gen_LDP}), large deviation results for the zig-zag process (Section \ref{sec:LDP_zigzag}) and an explicit expression of the rate function associated with the zig-zag process (Section \ref{sec:rate}). All proofs are deferred to Section \ref{sec:proofs}. 

\section{Preliminaries}
\label{sec:prelim}

\subsection{Notation and definitions}
\label{sec:notation}
Throughout the paper, $E$ will denote a complete separable metric space and $\calB (E)$ the relevant $\sigma$-algebra on $E$; unless otherwise stated this is taken to be the Borel $\sigma$-algebra. $C(E)$ and $C_b(E)$ are the spaces of functions $f : E \to \bR$ that are continuous and bounded continuous, respectively. The space of continuous and right-continuous functions from $[0, \infty)$ to $E$ is denoted by $C_E[0, \infty)$ and $D_E [0,\infty)$, respectively.  A sequence of functions $\{ f_n \}_n$ on $E$ converges {\it boundedly and uniformly on compacts} to a function $f$ if and only if $\sup _n \norm{f_n} < \infty$ and for each compact $K \subseteq E$,
	\begin{align*}
		\lim _{n\to \infty} \sup_{x \in K} | f_n(x) - f(x) | =0.
	\end{align*}
	This is denoted as $f = buc-\lim_{n\to \infty} f_n$.
	
	For a Markov process $Y = \{ Y_t: \ t \geq 0\}$, we denote by $S = \{ S(t): t \geq 0 \}$ the associated Markov semigroup. A semigroup $S(t)$ acting on $C(E)$ is {\it Feller continuous} if, for any $t$, $S(t) : C(E) \to C(E)$, {\it strongly continuous} if $S(t) f \to f$ as $t \to 0$ for any $f \in C(E)$ and {\it buc-continuous} if $buc-\lim _{t\to 0} S(t) f = f$ for $f \in C_b(E)$. 
	
	For an operator $L$, $\calD (L)$ denotes the domain of $L$. For functions in $\calD (L)$, $\calD ^+ (L)$ denotes those that are strictly positive and $\calD^{++}(L)$ those that are positive and uniformly bounded from below by a positive constant. For a given $L$ we use $B$ to denote the extended generator associated with $L$.


We use $\mathcal{P}(E)$ to denote the space of probability measures on $E$, and $\calP _c (E)$ is the subset of probability measures with compact support. Throughout the paper we equip $\calP (E)$ with the topology of weak convergence: $\rho _n \to \rho$ in this topology if
\begin{align*}
	\int _E f(x) \rho _n (dx) \to \int _E f(x) \rho (dx), \ n\to \infty, \ \ \forall f \in C_b (E).
\end{align*}
 A special case that will be considered several times is $\mathcal{P}(D_E [0,\infty))$, which is also equipped with the weak topology. For a process $\{ Y(t), t \geq 0 \}$ taking values in $E$ and $y \in E$, we denote by $\mathbb{P}_y \in \mathcal{P}(D_E [0,\infty) )$ the distribution of the process $Y(t)|_{t \geq 0}$ starting at~$y \in E$.
 
The set of positive Borel measures on $E$ is denoted by $\calM (E)$ and the set of finite Borel measures on $E$ are denoted by $\calM _f (E) \subset \calM (E)$. We let $\calL (E)$ denote the following subset of $\calM (E \times [0,\infty))$:
\begin{align*}
	\calL (E) = \{ z \in \calM (E \times [0, \infty)): \ z(E \times [0,t]) = t, \ t \geq 0 \}. 
\end{align*}
The set $\calL (E)$ is endowed with the topology of weak convergence on bounded time intervals: for $\{ \rho_n \} \subset \calL (E)$, $\rho_n \to \rho$ if for all $f \in C_b(E \times [0, \infty))$ and all $t \geq 0$,
\begin{align*}
	\int _{E \times [0,t]} f(x,s) d\rho_n(x,s) \to \int _{E \times [0,t]} f(x,s) d\rho(x,s).
\end{align*}
Then $\calL (E)$ is the set of Borel-measures on $E \times [0, \infty)$  of the form 
\begin{align*}
	d\rho (x,t) = \mu _t (dx) dt,
\end{align*}
for probability measures $\mu_t \in \calP (E)$. That is , for every $\rho \in \calL (E)$, there exists a measurable path $s \mapsto \mu_s \in \calP (E)$ such that 
\begin{align*}
	\rho (A \times [0,t] ) = \int _0 ^t \mu _s (A)ds, \ \textrm{ for any } A \in \calB (E), \ t >0.
\end{align*}

For two real-valued functions $f$ and $g$, $f \sim g$ as $x \to x_0$ means that they are asymptotically equivalent in the limit $x \to x_0$:
\begin{align*}
	\frac{f(x)}{g(x)} \to 1, \ \ \textrm{as } x \to x_0. 
\end{align*}


\subsection{Large deviations for empirical measures}
\label{sec:intro_LDP}
Consider a Markov process $Y = \{ Y_t: t \geq 0 \}$ taking values in a complete separable metric space $E$, with associated generator $L : \mathcal{D}(B) \subseteq C_b(E) \to C_b(E)$ and semigroup $S(t)$. 
The {\it empirical measure} $\eta_t$ associated with $Y_t$ is the stochastic process with values in $\mathcal{P}(E)$ defined by
\begin{equation*}
\eta_t (A) = \frac{1}{t} \int_0^t \boldsymbol{1}_A(Y_s)ds,\quad A \in \mathcal{B}(E).
\end{equation*}
%
%
Empirical measures play an important r\^ole in, for example, the settings of MCMC methods and steady-state simulations, via the pairing of measures and observables: For a probability measure $\mu \in \mathcal{P}(E)$ and a function $V \in C_b(E)$, we write
\begin{equation*}
\mu(V) = \int_E V(y) d\mu(y)
\end{equation*}
for the pairing of measures and observables. For the empirical measure $\eta_t$, this pairing corresponds to time averages,
\begin{equation}
\label{eq:etaV}
\eta_t(V) = \frac{1}{t}\int_0^t V(Y_s) ds.
\end{equation}
%
%
If there is an invariant measure $\pi \in \mathcal{P}(E)$ associated with the generator $L$, ergodicity of the process $Y_t$ will ensure the convergence $\eta _t \to \pi$ as $t\to \infty$, w.p.\ 1 in $\calP (E)$, from which it follows that for any $V \in C_b(E)$,
\begin{align*}
	\eta_t(V) \to \pi(V)	\quad \text{ as } t \to \infty, \; \mathbb{P}-a.s.
\end{align*}
Thus, time averages such as \eqref{eq:etaV} are precisely what is used to form approximations in Monte Carlo methods and there is a direct link between the performance of such simulation methods and the properties of the empirical measure.

The theory of large deviations for empirical measures is concerned with deviations of $\eta _t$ from $\pi$ as $t$ grows large. The gist of the so-called large deviations principle is that for any $\rho \in \mathcal{P} (E)$, for large $t$
\begin{align*}
	\bP_y (\eta _t \approx \rho) \simeq \exp \left\{ -t I(\rho) \right\},
\end{align*}
where the function $I : \calP (E) \to [0,\infty]$ is the rate function associated with the process. The formal definition underlying all of large deviation theory, making the previous display rigorous, is as follows:
	The sequence $\{ \eta _t \} _{t \geq 0}$ satisfies a {\it large deviations principle (LDP)} with speed $t$ and {\it rate function} $I$, if $I$ is lower semicontinuous, has compact sublevel-sets and for any measurable set $A \subseteq \calP (E)$,
	\begin{align*}
		- \inf _{\mu \in A^\circ} I (\mu) &\leq \liminf _{t\to \infty} \frac{1}{t} \log \bP \left(\eta _t \in A^\circ \right) \\
		&\leq \limsup _{t\to \infty} \frac{1}{t} \log \bP \left( \eta _t \in \bar A \right) \leq -\inf _{\mu \in \bar A} I(\mu),
	\end{align*}
where $A^\circ$ and $\bar A$ is the interior and closure of the set $A$.

Under relatively mild conditions on the dynamics of the process $Y$ the rate function will be strictly convex and satisfy $I(\mu) =0$ if and only if $\mu = \pi$. Thus, the rate function characterises the exponential rate of decay of probabilities of sets not including the invariant distribution $\pi$. Moreover the rate function can be used to characterise {\it how} events may occur - for sets $A$ that do not include $\pi$, the minimisers of $I$ over $A$ represent the behaviour $\eta _t$ is most likely to exhibit if $A$  occurs. 

For empirical measures of Markov processes, the rate function associated with an LDP can often be expressed using a variational form, obtained by Donsker \& Varadhan \cite{donsker1975asymptoticI}, involving the generator $L$ of the underlying process. For the compact setting, they proved the following result in \cite{donsker1975asymptoticI}.

\begin{theorem}[{\cite[Theorem 3]{donsker1975asymptoticI}}]\label{thm:proving_LDP:DV_I}
Take $E$ to be a compact, complete separable metric space. Let $S(t)$ be a Markov semigroup acting on $C(E)$ equipped with the supremum norm, $L$ is the generator associated to $S$. Assume the following:
\begin{enumerate}[label =(DV.\arabic*)]
\item\label{item:thm_DV_I:Feller} The semigroup is Feller continuous and strongly continuous. 
\item\label{item:thm_DV_I:reference_measure} There exists a probability measure $\lambda \in \mathcal{P}(E)$ such that for each $t > 0$ and $x\in E$, the transition probabilities $P(t,x,dy)$ are absolutely continuous with respect to $\lambda$, that is 
\begin{equation*}
P(t,x,dy) = p(t,x,y) \lambda(dy),
\end{equation*}
for some $p$ with $0< a(t) \leq p(t,x,y) \leq A(t) < \infty$.
\end{enumerate}
Then the associated sequence $\{\eta_t\}_{t > 0}$ satisfies a large deviation principle in $\mathcal{P}(E)$, with rate function $\mathcal{I}: \mathcal{P}(E) \to [0,\infty]$ given by
\begin{equation}\label{eq:proving_LDP:DV_rate_function}
\mathcal{I}(\mu) = -\inf_{u \in \mathcal{D}^+(L)}\int_E \frac{Lu}{u} d\mu.
\end{equation}
\end{theorem}
%
The theorem applies in particular to drift-diffusions taking values in a compact space. Roughly speaking, for such processes, with reasonable coefficients, the Feller-continuity is satisfied and the diffusive part ensures absolute continuity with respect to a volume measure $dx$. In \cite{rey2015irreversible} Rey-Bellet and Spiliopoulos use this result to study performance of specific non-reversible MCMC methods based on drift-diffusions; their Assumption~(H) allows for an application of Theorem \ref{thm:proving_LDP:DV_I}. 

%
Condition~\ref{item:thm_DV_I:reference_measure} is a reasonable transitivity assumption for processes that involve a diffusive term. However, this condition excludes many interesting examples, such as continuous-time jump processes, see e.g.\ \cite{dupuis2015large}. 
The issues highlighted therein are present also for the zig-zag process on $\bR \times \{ \pm 1\}$: in a sense, the absence of a diffusive operator excludes the possibility of finding a suitable reference measure.
%
%

If the process $Y_t$ is reversible with respect to the reference measure, that is $p(t,x,y) = p(t,y,x)$, then the rate function takes a more explicit form, see e.g.\ Theorem~5 in~\cite{donsker1975asymptoticI}. However, our interests are explicitly in non-reversible processes, such as the zig-zag process, and therefore such representations are not available. 
%
%

In conclusion, while Theorem \ref{thm:proving_LDP:DV_I} can be a starting point for many drift-diffusion processes, it is not a sufficient tool for many other interesting processes, including the position-velocity PDMPs. In order to use large deviation results to study performance of such MCMC algorithms we must first overcome this obstacle and establish the relevant large deviations principles.



In~\cite{dupuis2018large}, Dupuis and Lipshutz consider large deviations of empirical measures of $\mathbb{R}^d$-valued drift-diffusions. Their Condition~2.2 corresponds to a type of stability criterion in terms of a Lyapunov function. A transitivity property similar to Condition~\ref{item:thm_DV_I:reference_measure} of Theorem~\ref{thm:proving_LDP:DV_I} is satisfied due to the diffusive part, and they prove a different, explicit representation of the rate function, assuming only standard regularity conditions on the coefficients. In particular, this representation holds for non-reversible drift-diffusions. 

\subsection{The zig-zag process}
\label{sec:zig-zag}

In this section we will discuss very concisely the zig-zag process. As discussed in the introduction the zig-zag process is an example of a piecewise deterministic Markov process \cite{davis1984piecewise}. As the name indicates, a piecewise deterministic Markov process is a Markov process with deterministic trajectories, in between event times at which the process makes a discontinuous change. 

For the (one-dimensional) zig-zag process, the state space is either $E = \R \times \{\pm 1\}$ or $E = \mathbb{T} \times \{\pm 1\}$ and a typical state is denoted in this paper by $(x,v)$. Here $x$ represents a position and $v$ a velocity. Starting from $(x,v)$ at time $t= T_0 := 0$, the dynamics of a Markov process $(X_t, V_t)$ are given, until the first (random) event time $T_1 > 0$,  by \[ (X_t, V_t) = (x +t v, v), \quad 0 \leq t < T_1.\] In other words, the position changes according to the constant velocity $v$, which itself does not change in between event times. The random time $T_1$ at which the first event happens is distributed according to
\[ \P_{x,v}(T_1 \geq t) = \exp \left( -\int_0^t \lambda(X_s, V_s) \, d s \right) = \exp \left( -\int_0^t \lambda(x + vs, v) \, d s \right),\]
where $\lambda : E \rightarrow [0,\infty)$ is the \emph{event rate}, which is in the case of the zig-zag process also known as the \emph{switching rate}, which we will discuss in more detail below. At an event time $T$ the velocity changes sign and the position remains unchanged: 
\[ V_{T_1} = -V_{T_1-} \quad \text{and} \quad X_{T_1} = X_{T_1-}.\]
From the time $T_1$ onward, the process repeats the dynamics described above: for $i = 1, 2, \dots$
\begin{align*} & X_t = X_{T_{i-1}}  +(t- T_{i-1}) V_{T_{i-1}}, \quad V_t = V_{T_{i-1}}, \quad T_{i-1}\leq t < T_{i}, \\
& \P(T_i \geq t \mid T_{i-1}, X_{T_{i-1}}, V_{T_{i-1}}) = \exp \left( -\int_{T_{i-1}}^t \lambda(X_{T_{i-1}} + s V_{T_{i-1}}, V_{T_{i-1}}) \, d s \right), \\
& X_{T_i} = X_{T_i-}, \quad V_{T_i} = -V_{T_i-}.
\end{align*}
The switching rate $\lambda : E \rightarrow \R$ is assumed to be continuous. If $\lambda$ satisfies
\begin{equation} \label{eq:switching-intensity-condition-1} \lambda(x, 1) - \lambda(x,-1) = U'(x), \end{equation} for a continuously differentiable function $U$, then the measure $\pi(d x, d v) = \exp(-U(x)) \, d x \otimes \unif_{\pm 1}(dv)$ is a stationary measure for $(X_t, V_t)$.  An equivalent condition to~\eqref{eq:switching-intensity-condition-1} is that, for some continuous non-negative function $\gamma(x)$ we have
\begin{equation}
\label{eq:switching-intensity-condition-2}    \lambda(x,v) = \max (0, v U'(x)) + \gamma(x).
\end{equation}
Here $(\max 0, vU'(x))$ is called the \emph{canonical switching intensity}, and $\gamma$ is called the \emph{excess switching intensity} or \emph{refreshment rate}.
We study the dependence of the empirical measure of the process $(X_t, V_t)$ on $\gamma$ in Section~\ref{sec:rate}.

The zig-zag process can be extended in a natural way to a multi-dimensional process in $\R^d \times \{\pm 1\}^d$; see \cite{bierkens2019zig, bierkens2019ergodicity}. Since we focus in this paper on properties of the one-dimensional process we will not discuss this extension here. The ergodic properties of the zig-zag process are essential in order to establish a large deviations theory for the empirical measure. Under mild conditions it can be shown that the zig-zag process is (exponentially) ergodic; see \cite{bierkens2017piecewise} for the one-dimensional case, and \cite{bierkens2019ergodicity} for the multi-dimensional zig-zag process.

By \cite[Theorem 26.14]{Davis1993}, the extended generator of the zig-zag process is given by
\[ B f(x,v) = v \partial_x f(x,v) + \lambda(x,v) [ f(x,-v) - f(x,v)], \quad (x,v) \in E,\]
with 
\[ \mathcal D(B) = \{ f : E \rightarrow \R : f(\cdot,v)  \, \text{is absolutely continuous for } v = \pm 1\}.\] 

\section
{Large deviations for empirical measures of PDMPs}
\label{sec:main}

In this section we present our main results: we establish a large deviations principle for the empirical measure of a Markov process under fairly general assumptions which include in particular examples of position-velocity PDMPs such as the zig-zag process. After obtaining these general results we focus for concreteness on the zig-zag process, for which we verify the stated assumptions. We also give an explicit characterisation of the corresponding rate function, a necessary step towards using the LDP for analysing the performance and properties of approximations based on the zig-zag process. To streamline the presentation we split the analysis according to whether we consider a compact or non-compact state space $E$. 
\smallskip

To facilitate the proof of the LDP for the empirical measures, we first formulate in Section~\ref{sec:gen_LDP} two more general large deviations results (compact and non-compact setting) for empirical measures arising from certain continous-time stochastic processes. We then show that the zig-zag process is a special case in this class of processes in Section \ref{sec:LDP_zigzag}. It is worth to emphasise that we do not aim for greatest generality in the large deviations results Theorems \ref{thm:proving_LDP:LDP_compact} and \ref{thm:proving_LDP:LDP_non_compact}. Rather, we settle for conditions that make the general conditions of Lemma \ref{lemma:FengKurtz} more transparent and concrete whilst still allowing us to prove the large deviations principle for the empirical measures of the zig-zag process.

\subsection{Results in an abstract setting aimed at position-velocity PDMPs}
\label{sec:gen_LDP}

Before we specialize to the zig-zag process, we consider the setting described in Section~\ref{sec:notation} to PDMPs: $Y$ is a Markov process taking values in a locally compact complete separable metric space $E$, 
with associated semigroup $S(t)$ and 
infinitesimal generator $L$. We also make use of the extended generator $B$; see \cite{ethier2009markov} and Section~\ref{sec:zig-zag}. Typically, $E=\mathbb{R}^d\times\mathcal{S}$ where $\mathbb R^d$ is the state space for a position variable $X_t$ and $\mathcal{S}$ is a compact set that models the state space of the velocity variables $V_t$. For the zig-zag process, $\mathcal{S}=\{\pm 1\}^d$, and for the Bouncy Particle Sampler $\mathcal S$ can be taken to be the $(d-1)$-dimensional unit sphere. Note that for $d =1 $ these two choices coincide.
\smallskip

The following are the assumptions we will impose in order to establish an LDP for the empirical measures of the process $Y$. Not all conditions are required at the same time: we impose conditions~\ref{item:thm_LDP_compact:Feller}, \ref{item:thm_LDP_compact:tight} and~\ref{item:thm_LDP_compact:principal_eigenvalue} for the compact case and \ref{item:thm_LDP_compact:Feller}, \ref{item:thm_LDP_compact:tight}, \ref{item:thm_LDP_non_compact:Lyapunov} and~\ref{item:thm_LDP_non_compact:mixing} for the non-compact case.
\begin{enumerate}[label = (A.\arabic*)]
\item\label{item:thm_LDP_compact:Feller} The semigroup $S(t)$ is a
Feller semigroup.
\item\label{item:thm_LDP_compact:tight} For any compact set $K\subseteq E$, the set of measures $\{\mathbb{P}_y: y \in K\}$ is tight in $\mathcal{P}(D_E[0,\infty))$.
\item\label{item:thm_LDP_compact:principal_eigenvalue} For any function $V \in C(E)$, there exists a function $u \in \mathcal{D}^+(L)$ and a real eigenvalue $\beta \in \mathbb{R}$ such that pointwise on $E$,
\begin{equation*}
(V+L) u = \beta u.
\end{equation*}
\item\label{item:thm_LDP_non_compact:Lyapunov} There exist two non-negative functions $g_1,g_2 \in C(E;[0,\infty))$ such that: 
	\begin{enumerate}
		\item\label{item:thm_LDP_non_compact:Lyapunov1} 
		For any $\ell \geq 0$, the sublevel-sets $\{g_i\leq \ell\}$ are compact and $g_i(y)\to\infty$ as $|y|\to\infty$,
		\item\label{item:thm_LDP_non_compact:Lyapunov2} 
		$g_1(y)/g_2(y) \to 0$ as $|y|\to\infty$,
		\item\label{item:thm_LDP_non_compact:Lyapunov3} $e^{g_i} \in \mathcal D(B)$, for any $c\in\mathbb{R}$ the superlevel-sets $\{y\in E\,:\,e^{-g_i(y)} B(e^{g_i})(y)\geq c\}$ are compact, and $e^{-g_1(y)} B(e^{g_1})(y)\to-\infty$ as $|y|\to\infty$,
	\end{enumerate}
	where we recall that $B$ is the extended generator of $Y$. We write $|y_n|\to\infty$ if $d(y_n,z)\to\infty$ for any point $z\in E$.
	\item\label{item:thm_LDP_non_compact:mixing} For any two compactly supported probability measures $\nu_1,\nu_2 \in \mathcal{P}_c(E)$, there exist constants $T,M > 0$ and measures $\rho_1,\rho_2 \in \mathcal{P}([0,T])$ such that for all Borel sets $A \subseteq E$,
	\begin{equation}\label{eq:proving_LDP:LDP_non_compact:mixing}
	\int_0^T\int_E P(t,y,A)\, d\nu_1(y)d\rho_1(t) \leq M \int_0^T\int_E P(t,z,A)\, d\nu_2(z)d\rho_2(t),
	\end{equation}
	where $P(t,y,dy')$ denotes the transition probabilities associated to $Y_t$.

\end{enumerate}
Conditions \ref{item:thm_LDP_compact:Feller}-\ref{item:thm_LDP_compact:principal_eigenvalue} are enough to prove Theorem \ref{thm:proving_LDP:LDP_non_compact}, the large deviations principle in a compact setting. In this setting conditions \ref{item:thm_LDP_compact:Feller} and \ref{item:thm_LDP_compact:tight} replace Condition~\ref{item:thm_DV_I:Feller} of Theorem \ref{thm:proving_LDP:DV_I};  Condition \ref{item:thm_LDP_compact:tight} can also be weakened to $\mathbb{P}_{y_n} \to \mathbb{P}_y$ in $\mathcal{P}(D_E[0,\infty))$ whenever $y_n \to y$. Together Conditions \ref{item:thm_LDP_compact:Feller} and \ref{item:thm_LDP_compact:tight} imply strong continuity of the semigroup $S$ (see e.g. \ \cite[Remark~11.22]{FengKurtz06}). 
\smallskip

As pointed out in Section \ref{sec:intro_LDP}, the processes we have in mind do not satisfy a transitivity condition similar to Condition \ref{item:thm_DV_I:reference_measure} of Theorem \ref{thm:proving_LDP:DV_I}. In the compact setting this can be replaced by condition~\ref{item:thm_LDP_compact:principal_eigenvalue}, which corresponds to a principal-eigenvalue problem for the operator $L+V$. In compact settings, such eigenvalue problems can usually be solved if the coefficients of the generator are regular enough. In Section \ref{sec:LDP_zigzag} we show that this is the case for the zig-zag process taking values in the compact torus. 
\smallskip

In the non-compact setting, the eigenvalue problem \ref{item:thm_LDP_compact:principal_eigenvalue} is replaced by conditions \ref{item:thm_LDP_non_compact:Lyapunov} and \ref{item:thm_LDP_non_compact:mixing}.
%
Condition~\ref{item:thm_LDP_non_compact:Lyapunov} is closely related to the stability conditions assumed in~\cite{donsker1976asymptotic} and~\cite{dupuis2018large}. Because $e^{g_1}$ is unbounded, formally we have to use the extended generator $B$ instead of the infinitesimal generator $L$ to formulate Condition~\ref{item:thm_LDP_non_compact:Lyapunov3}. The same problem occurs in Condition~2.2 of~\cite{dupuis2018large}: for a diffusion process $Y_t$ in $\mathbb{R}^d$ satisfying
$
dY_t = -Y_t dt + dW_t,
$
the second-order differential operator 
\begin{equation*}
Bf (x) = \frac{1}{2} \Delta f(x) - x\nabla f(x)
\end{equation*}
acting on $C^2(\mathbb{R}^d)$ is well-defined and is equal to the infinitesimal generator $L$ of the process when restricted to $C^2_b(\mathbb{R}^d)$. With $g_1(x) = \delta |x|^2/2$, the function $e^{-g_1(x)}Be^{g_1}(x)$ goes to minus infinity for $\delta$ small enough. A second Lyapunov function is $g_2(x) = \sqrt{1+|x|^2}$. In the context of the zig-zag process, since $E$ is of the form $E=\mathbb{R}^d\times\{\pm 1\}^d$, using continuous functions that grow to infinity when fixing the velocity variable is sufficient for obtaining compact level sets.
\smallskip

Condition~\ref{item:thm_LDP_non_compact:mixing} plays the role of a transitivity assumption in the non-compact case. While it is feasible to solve a principal-eigenvalue problem for a compact state space, this is much more difficult in the non-compact setting. It would require deriving not only the eigenvalue itself, but also the corresponding eigenfunction on a non-compact space, for which general existence results are not available. 
In this setting the transitivity condition \ref{item:thm_DV_I:reference_measure} of Theorem \ref{thm:proving_LDP:DV_I} is instead partly replaced by the mixing property \ref{item:thm_LDP_non_compact:mixing}. It is a slightly weakened version of Condition B.8 in \cite{FengKurtz06}, in that it requires the transition probabilities to be comparable only for compactly supported initial conditions $\nu_1,\nu_2 \in \mathcal{P}_c(E)$. This weakening is crucial for the results in this paper, because the stronger condition fails to be true for the zig-zag process if for instance $\nu_1 = \mathcal{N}(0,1) \otimes \unif_{\pm 1}$ and $\nu_2 = \delta_y$. In that example, while the left-hand side of \eqref{eq:proving_LDP:LDP_non_compact:mixing} is in this case positive for any Borel set $A \subseteq E$, the right-hand side can become zero. This is because the zig-zag process has finite speed propagation, so that for arbitrary $T > 0$, if $\mathrm{dist}(A,y) > T$, then the probability of transitioning from $y$ into $A$ is zero. However for compactly supported measures the condition is satisfied. We verify the conditions of Theorem \ref{thm:proving_LDP:LDP_non_compact} for the zig-zag process in Section \ref{sec:LDP_zigzag}. 
\smallskip

We are now ready to state the two general large deviations results of this paper, which in Section \ref{sec:LDP_zigzag} will be used to derive the large deviations principle for the empirical measures of the zig-zag process. We start with the compact setting.
\begin{theorem}\label{thm:proving_LDP:LDP_compact}
Let $E$ be compact, $S(t)$ a Markov semigroup acting on $C(E)$ equipped with the supremum norm, and $Y_t$ the corresponding Markov process. Let $L$ be the infinitesimal generator of $Y_t$, and assume that $Y_t$ solves the associated martingale problem. Suppose Assumptions \ref{item:thm_LDP_compact:Feller}, \ref{item:thm_LDP_compact:tight} and \ref{item:thm_LDP_compact:principal_eigenvalue} hold. Then 
the empirical measures $\{\eta_t\}_{t > 0}$ associated to $Y_t$ satisfy a large deviations principle in $\mathcal{P}(E)$ with rate function $\mathcal{I}:\mathcal{P}(E) \to [0,\infty]$ given by~\eqref{eq:proving_LDP:DV_rate_function}.
\end{theorem}

Theorem \ref{thm:proving_LDP:LDP_compact} remains valid when replacing the eigenvalue-problem condition~\ref{item:thm_LDP_compact:principal_eigenvalue} by the mixing condition~\ref{item:thm_LDP_non_compact:mixing}. This is because the latter is a weaker condition sufficient for verifying the inequality \eqref{eq:ineqH1H2}, upon which the proof of the theorem hinges.

The next theorem gives the corresponding large deviations result for the non-compact setting; this is the result we use for proving the large deviations principle for the zig-zag process on $\bR \times \{\pm 1\}$ (Theorem \ref{thm:LDP-zigzag:non-compact-1d}).
\begin{theorem}\label{thm:proving_LDP:LDP_non_compact}
Let 
$S(t)$ be a Markov semigroup acting on $C_b(E)$ and $Y_t$ the corresponding Markov process. Let $L$ be the infinitesimal generator of $Y_t$ and assume that $Y_t$ solves the associated martingale problem. Assume \ref{item:thm_LDP_compact:Feller}, \ref{item:thm_LDP_compact:tight},
\ref{item:thm_LDP_non_compact:Lyapunov} and \ref{item:thm_LDP_non_compact:mixing}.
Then, if $Y_0 \in K$ for some compact set $K$, the empirical measures $\{\eta_t\}_{t > 0}$ associated to the Markov process $Y_t$ satisfy a large deviations principle in $\mathcal{P}(E)$, with rate function $\mathcal{I} : \mathcal{P}(E) \to [0,\infty]$ given by
\begin{equation*}
\mathcal{I}(\mu) = -\inf_{u \in \mathcal{D}^{++}(L)} \int_E \frac{Lu}{u} d\mu.
\end{equation*}
%
\end{theorem}

The proofs of Theorems \ref{thm:proving_LDP:LDP_compact} and \ref{thm:proving_LDP:LDP_non_compact} are given in Sections ~\ref{section:proof-of-LDP:compact} and \ref{section:proof-of-LDP:non-compact}, respectively.

%
\subsection{The empirical measures of the zig-zag process}
\label{sec:LDP_zigzag}
Having established the general large deviations results Theorems \ref{thm:proving_LDP:LDP_compact} and \ref{thm:proving_LDP:LDP_non_compact}, we now specialize to the zig-zag process. Throughout the section, $Y_t$ is used to denote the zig-zag process, $Y_t = (X_t, V_t)$ with $X_t$ and $V_t$ as in Section \ref{sec:zig-zag}. However the state space $E$ will change as we split the large deviations statements for the empirical measures of $Y$ into compact (torus) and non-compact ($\bR$) settings. Although \ref{thm:proving_LDP:LDP_non_compact} holds for arbitrary dimension $d \geq 1$, for the zig-zag process we limit ourselves to verifying the conditions for the case $d = 1$. Extending these results to $d > 1$ is substantially more difficult and remains a topic of further research. While conditions~\ref{item:thm_LDP_compact:Feller}, \ref{item:thm_LDP_compact:tight}, \ref{item:thm_LDP_non_compact:Lyapunov} hold true, the main challenge is verifying~\ref{item:thm_LDP_non_compact:mixing}.


We begin by considering the compact state space  $\mathbb{T} \times\{\pm 1\}$. In this case the infinitesemal generator $L$ of the semigroup $S(t)$ is has domain $\mathcal{D}(L) = C^1(\bT \times \{ \pm 1\}) = \{f\in C(\bT \times \{ \pm 1\}): f(\cdot,\pm 1) \in C^1(\mathbb{T})\}$,
and takes the form
\begin{equation}
\label{eq:gen_compact}
Lf (x,v) = v \partial_x f(x,v) + \lambda(x,v) \left[f(x,-v) - f(x,v)\right],
\end{equation}
with $\lambda$ given by \eqref{eq:switching-intensity-condition-2}. The LDP for the empirical measures associated with $Y$ and this state space is given in Theorem \ref{thm:LDP-zigzag:compact}. We prove this result in Section~\ref{section:zig-zag:LDP:compact} by verifying the conditions of Theorem \ref{thm:proving_LDP:LDP_compact}, the large deviations principle for processes taking values in a compact state space.


\begin{theorem}\label{thm:LDP-zigzag:compact}
Suppose that $U \in C^2(\mathbb{T})$. Then the family of empirical measures $\{\eta_t\}_{t>0}$ of the zig-zag process taking values in $\bT \times \{\pm 1\}$ satisfies a large deviations principle in the limit $t \to \infty$, with rate function $\mathcal{I}:\mathcal{P}(\bT \times \{\pm 1\}) \to [0,\infty]$ given by
\begin{equation*}
\mathcal{I}(\mu) = -\inf_{u \in \mathcal{D}^+(L)} \int_{\bT \times \{\pm 1\}} \frac{Lu}{u} d\mu.
\end{equation*}
%
%
\end{theorem}
We now move to the setting of a non-compact state space. Specifically, we consider the zig-zag process $Y_t =(X_t,V_t) $ taking values in $\bR \times \{ \pm 1 \}$. As before, $L$ is the generator of this process, i.e.\ $L : \mathcal{D}(L) \subseteq C_b(\bR \times \{ \pm 1\}) \to C_b(\bR \times \{ \pm 1\})$ is a densely defined linear operator, on the set of functions $\{ f(\cdot, \pm 1) \in C_b ^1(\bR) \}$ we have the representation
\begin{align*}
	Lf(x,v) = v \partial _x f(x,v) + \lambda(x,v) \left[ f(x, -v) - f(x,v) \right], \ \ f \in \calD (L),
\end{align*}
with $\lambda(x,v)=\max(0,vU'(x))+\gamma(x)$. To prove the large deviations principle in this non-compact setting we need additional assumptions on the potential function $U$ determining the jump rates.
\begin{enumerate}[label =($B$.\arabic*)]
	\item\label{item:thm-LDP-zig-zag:Ui} $U(x) \to \infty$ as $|x| \to \infty$ and $U'(x) \to \pm \infty$ as $x \to \pm \infty,$
	\item\label{item:thm-LDP-zig-zag:Uii} $U'(x)/U(x) \to 0$ as $|x| \to \infty,$
	\item\label{item:thm-LDP-zig-zag:Uiii} $U''(x)/U'(x) \to 0$ as $|x| \to \infty.$
\end{enumerate}
Furthermore, we will assume that there exists a second potential $V \in C^2(\mathbb{R})$ such that:
\begin{enumerate}[label =($C$.\arabic*)]
	\item\label{item:thm-LDP-zig-zag:Upsi} $V(x) \to \infty$ as $|x| \to \infty$ and $V'(x) \to \pm \infty$ as $x \to \pm \infty$,
	\item\label{item:thm-LDP-zig-zag:Upsii} $V(x)/U(x) \to 0$, $U'(x)/V(x) \to 0 $ and $V'(x)/U'(x) \to 0$ as $|x| \to \infty$,
	\item\label{item:thm-LDP-zig-zag:Upsiii} $U''(x)/V'(x) \to 0$ as $|x| \to \infty$.
\end{enumerate}
In Section~\ref{section:zig-zag:LDP:non_compact}, we prove the following Theorem.
\begin{theorem}\label{thm:LDP-zigzag:non-compact-1d}
Assume that $U\in C^3(\mathbb{R})$ satisfies \ref{item:thm-LDP-zig-zag:Ui} - \ref{item:thm-LDP-zig-zag:Uiii}, that there is a function $V \in C^2(\mathbb{R})$ satisfying \ref{item:thm-LDP-zig-zag:Upsi} - \ref{item:thm-LDP-zig-zag:Upsiii} and the function $\gamma$ in \eqref{eq:switching-intensity-condition-2} is uniformly bounded by some $\bar \gamma$.
%
%
Suppose that the initial condition $Y_0$ belongs to a compact set $K \subseteq \bR \times \{ \pm 1\}$. Then the empirical measures $\{\eta_t\}_{t > 0}$ of $Y$ satisfies a large deviations principle on $\mathcal{P}(\bR \times \{ \pm 1\})$ with speed $t$ and rate function $\mathcal{I}:\mathcal{P}(\bR \times \{ \pm 1\}) \to [0,\infty]$ given by
\begin{equation*}
\mathcal{I}(\mu) = -\inf_{u \in \mathcal{D}^{++}(L)} \int_{\bR \times \{ \pm 1\}} \frac{Lu}{u} d\mu.
\end{equation*}
%
\end{theorem}
Some comments on the additional assumptions \ref{item:thm-LDP-zig-zag:Ui} - \ref{item:thm-LDP-zig-zag:Uiii} and \ref{item:thm-LDP-zig-zag:Upsi} - \ref{item:thm-LDP-zig-zag:Upsiii} are in place. The condition $U \in C^3(\mathbb{R})$ is imposed to allow for an application of Theorem 4 of \cite{bierkens2019ergodicity}, which is used to verify that \ref{item:thm_LDP_non_compact:mixing} holds. 
The auxiliary potential $V$ is used to find a second Lyapunov function for $L$ that grows slower than $U$ at infinity; roughly speaking, $V$ behaves asymptotically in-between the potential $U$ and its derivative $U'$ as $|x|$ grows. As an example, in the Gaussian case, $U(x) = x^2/2$ satisfies Conditions \ref{item:thm-LDP-zig-zag:Ui} - \ref{item:thm-LDP-zig-zag:Uiii}, and for any $0<\kappa<1$, the potential $V(x) = |x|^{1+\kappa}/(1+\kappa)$ satisfies \ref{item:thm-LDP-zig-zag:Upsi} - \ref{item:thm-LDP-zig-zag:Upsiii}. In general, any potential $U$ growing at infinity as $(1+|x|^2)^{\beta/2}$ with $\beta > 1$ satisfies the conditions, with $0<\kappa < 1$ such that $\beta - \kappa > 1$ and auxiliary potential $V(x) \sim (1+|x|^2)^{(\beta - \kappa)/2}$.

\subsection{Explicit expression for the rate function}
\label{sec:rate}
In Theorems \ref{thm:LDP-zigzag:compact} and \ref{thm:LDP-zigzag:non-compact-1d} we establish the LDP for the empirical measures of the zig-zag process taking values in $\bT \times \{ \pm 1\}$ and $\bR \times \{ \pm 1 \}$, respectively. In those results the rate function is given on the variational form of the results by Donsker and Varadhan, see Section \ref{sec:intro_LDP}. This form follows from the more general large deviations results in Section \ref{sec:gen_LDP} and are not specific to the zig-zag process. Here, we use the properties of the latter to derive a more explicit form of the rate function for the case $E = \mathbb{T} \times \{\pm 1\}$, taking a first step towards using it as a tool for analysing the corresponding simulation algorithms. 


We assume throughout that $E = \mathbb{T} \times \{\pm 1\}$ and the switching rate $\lambda(x,v)$ satisfies $\lambda(x,v) > 0$ for all $(x,v) \in E$. This does not include the canonical rates $\lambda(x,v) = \max(0, v U'(x))_+$; however at the end of this section we present a formal expression for this case. 

Define a reference measure $\nu_0$ on $E$ by $\nu_0(dx, dv) = \Leb(dx) \otimes \unif_{\pm 1}(dv)$.
For any function $f : E \rightarrow \R$ we write $f^+(x) := f(x,+1)$ and $f^-(x) := f(x,-1)$.
Recall the $\arcsinh$ function,
\[ \arcsinh(\xi) = \log \left( \xi + \sqrt{\xi^2 +1} \right), \quad \xi \in \R.\]
The proofs of the following results are given in Section~\ref{sec:proofs-rate-function}.

\begin{proposition}
\label{prop:expression-ratefunction}
Suppose $\mu(dx, dv) = \rho(x,v) \nu_0(dx, dv)$ for a continuously differentiable function $\rho : E \rightarrow [0,\infty)$.
If $\frac{d \rho^+}{dx}(x)= \frac{d \rho^-}{dx}(x)$ and $\rho^+$, $\rho^-$ are strictly positive for all $x \in \mathbb{T}$ then the Donsker-Varadhan functional is given by
\begin{align}
\label{eq:rate-function}
\nonumber  \mathcal I(\mu)& = \int_{\mathbb{T}} \bigg\{\tfrac 1 2 \rho' \log \left( \frac{\lambda^+ \rho^+}{\lambda^- \rho^-} \right)+ \rho'  \arcsinh \left( \frac{ \rho'}{2 \sqrt{ \lambda^+ \lambda^- \rho^+ \rho^-}}\right) \\
 & \quad \quad \quad- \sqrt{ 4 \lambda^+ \lambda^- \rho^+ \rho^- + (\rho')^2} + \lambda^+ \rho^+ + \lambda^- \rho^- \bigg\} \ d x.
 \end{align}
If $\rho^+ \geq 0$ and $\rho^- \geq 0$ are constant, then
\begin{equation}
\label{eq:rate-function-constant-density}
\nonumber  \mathcal I(\mu) = \int_{\mathbb{T}} \left( \sqrt{\lambda^+ \rho^+} - \sqrt{\lambda^- \rho^-} \right)^2 \ d x.
 \end{equation}
If $\frac{d \rho^+(x)}{dx}(x) \neq \frac{d \rho^-(x)}{dx}$ for some $x \in \mathbb{T}$ then $\mathcal I(\mu) = \infty$.
\end{proposition}


Note that if $\mu(dx,dv) =  \rho(x,v) \nu_0(dx, dv)$ and $\frac{d \rho^+}{dx}= \frac{d \rho^-}{dx}$ on $\mathbb{T}$, then for some constant $c \in \R$ and a probability density function $\rho$ on $\mathbb{T}$ we have $\rho^+(x,v) = \rho(x) + c$ and $\rho^-(x,v) = \rho(x) - c$. 

A useful application of the rate function $\mathcal I(\mu)$ is in estimating deviations of ergodic averages, which typically requires the computation of $\inf_{\mu \in \mathcal P} \left(\mathcal I(\mu) - \int_E V \ d \mu \right)$. If $V$ does not depend on $v$, then by the following result we can safely assume $c = 0$ and thus restrict the minimization problem to minimization over probability densities on $\mathbb{T}$.

\begin{proposition}
\label{prop:no-v-dependence}
Let $\rho \in C^1(\mathbb{T})$ be a strictly positive probability density function on $\mathbb{T}$. Let $|k| := \inf_{x} \rho(x)$. Consider the one-parameter family of probability measures $(\mu_c)_{c \in (-k,+k)} \in \mathcal P_{eq}$ with probability density functions $\rho_c : E \rightarrow (0, \infty)$ given by
\[ \rho_c(x,+1) = \rho(x) + c, \quad \rho_c(x,-1) = \rho(x) - c, \quad c \in (-k, +k).\]
Then $c \mapsto \mathcal I(\mu_c)$ is minimized at $c = 0$. Furthermore, for $\mu = \mu_0$,  
\begin{equation}
\label{eq:ratefunction-simplified}
\begin{aligned}
  \mathcal I(\mu) 
 &  = \int_{\mathbb{T}} \bigg\{ \tfrac 1 2 \rho' \log \left( \frac{ \lambda^+}{\lambda^-} \right) + \rho' \arcsinh \left( \frac{\rho'}{2 \rho \sqrt{ \lambda^+ \lambda^-}} \right) \\
 & \quad \quad \quad  - \sqrt{ 4 \lambda^+ \lambda^- \rho^2 + (\rho')^2} + (\lambda^+ + \lambda^-) \rho \bigg\} \ dx.
 \end{aligned}
\end{equation}
\end{proposition}

We will specialize to the case in which $\rho = \rho^+ = \rho^-$ and use the representation $\rho = \exp(-W)$, where $W \in C^1(\mathbb{T})$. We  then find 
\begin{equation}
\label{eq:in-W}
\begin{aligned}
\mathcal I(\mu)
& = \int_{\mathbb{T}} \bigg\{ -\tfrac 1 2  W' \log \left( \frac{\lambda^+}{\lambda^-} \right) + W' \arcsinh \left( \frac{ W'}{2 \sqrt{\lambda^+ \lambda^-}} \right) \\
& \quad \quad \quad - \sqrt{4 \lambda^+ \lambda^- + (W')^2} + \lambda^+ + \lambda^- \bigg\} \exp(-W) \ d x.
\end{aligned}
\end{equation}
Let 
\begin{equation} \label{eq:switching-rate} \lambda^+(x) = \gamma + \max(0, U'(x)) \quad \mbox{and} \quad \lambda^-(x) = \gamma + \max(0, - U'(x)),\end{equation}
where $\gamma > 0$ is constant, so that $\lambda^{\pm}$ satisfy~\eqref{eq:switching-intensity-condition-2} and hence the measure with $\nu_0$-density $\exp(-U(x))$ is invariant. We call $\gamma$ the \emph{excessive switching intensity} or \emph{refreshment rate}.

We can now investigate the dependence of the rate function $\mathcal I$, through the expression~\eqref{eq:in-W}, on $\gamma$. 
The derivative of the integrand of~\eqref{eq:in-W}  with respect to $\gamma$ can be computed to be
\begin{equation} \label{eq:gamma-derivative} \left( \frac{4 \lambda^+ \lambda^- +(\lambda^+ - \lambda^-)W' - (\lambda^+ + \lambda^-)\sqrt{4 \lambda^- \lambda^+ + (W')^2}}{2 \lambda^+ \lambda^-}\right)\exp(-W),\end{equation}
which is non-positive, and zero only if $W' = \lambda^+ - \lambda^- = U'$ (which can be seen by maximizing with respect to $W'$).
It follows that $\mathcal I(\mu)$ is strictly decreasing as a function of $\gamma$ (for $\mu$ not equal to the stationary measure). In other words, for a smaller refreshment rate $\gamma$, the rate function increases, so that the convergence of empirical averages to equilibrium is faster.

Suppose that $\nu_0 (\{ x \in \mathbb{T} : U'(x) = 0\}) = 0$, i.e. the set of points where the derivative of $U$ vanishes is $\nu_0$-negligible. 
In the (formal) limit as $\gamma \downarrow 0$ in~\eqref{eq:switching-rate}, we obtain the following expression for $\mathcal I(\mu)$:
\begin{equation}
 \label{eq:limit-functional}
 \begin{aligned}
 \mathcal I(\mu) & = \begin{cases} \int_{\mathbb{T}} \left\{ |W'| \left( \log\left( \frac{ W'}{U'} \right) - 1 \right) + |U'| \right\}  \exp(-W) \ d x  \quad & \mbox{if $\operatorname{sign}(W') \equiv \operatorname{sign}(U')$}, \\
              \infty \quad & \mbox{otherwise}.
             \end{cases}
  \end{aligned}
\end{equation}


\section{Proofs}
\label{sec:proofs}

\subsection{Proofs of the general large deviations Theorems \ref{thm:proving_LDP:LDP_compact} and \ref{thm:proving_LDP:LDP_non_compact}}
\label{sec:aux}
In this section we give the proofs of Theorems \ref{thm:proving_LDP:LDP_compact} and \ref{thm:proving_LDP:LDP_non_compact}, which are used to obtain the large deviations principle for the empirical measures of the zig-zag process. The case of a compact state space is treated in Section \ref{section:proof-of-LDP:compact} and the non-compact case in Section \ref{section:proof-of-LDP:non-compact}. Before we embark on these proofs we outline the overall strategy; a more detailed description can be found in the book by Feng and Kurtz in \cite{FengKurtz06}. 

Consider the empirical measure
\begin{equation*}
\eta_t(\cdot) = \frac{1}{t} \int_0^t \delta_{Y_s} (\cdot) ds.
\end{equation*}
With a change of variable $s \mapsto ts$ in the integral we can express this as
\begin{align*}
	\eta_t (\cdot) = \int _0 ^1 \delta _{Y_{st} } (\cdot) ds,
\end{align*}
the empirical measure for the sped-up process (we can think of $t>1$) $Y^t _s = Y_{st}$ over the time interval $[0,1]$; in fact we will use $t=n \in \bN _+$ below. We can consider the empirical measure of this time-scaled process $Y_t$ on time intervals of lengths other than unity: for $\tau >0$ define $\eta _t ^{\tau}$ as
\begin{align*}
	\eta _t ^{\tau} (\cdot) = \int _0 ^{\tau} \delta _{Y_{st}} ds.
\end{align*}
This empirical measure is viewed as an element of $\calL (E)$, the set of Borel measures on $E \times [0, \infty)$ of the form $d\rho (x,s) = \mu_s (dx)ds$, $\mu _s \in \calP (E)$ (see Section \ref{sec:notation}). Any such $\rho \in \calL (E)$ defines a continuous path $t \mapsto \rho _t = \rho (\cdot \times [0,t]) \in \calM _f (E)$ and for $t=1$ this is a probability measure. 

The strategy for proving the large deviations principle for  $\{ \eta _t \}$ is to first show that $\{ \eta _t ^{\tau} \}$ satisfies a large deviations principle in $\calL (E)$. We can then use the fact that projections are continuous maps on $\calL (E)$ (Lemma~\ref{lemma:proofs_LDP:continuous_projection}) and an application of the contraction principle to obtain the sought-after large deviations principle on $\calP (E)$. This is summarised in the following proposition.
\begin{proposition}
\label{prop:LDP_strategy}
	Suppose that the family $\{\eta^\tau_t\}_{t > 0}$ satisfies a large deviations principle in $\mathcal{L}(E)$ with rate function $\mathcal{J}:\mathcal{L}(E) \to [0,\infty]$ given by
\begin{equation*}
\mathcal{J}(\rho) = \int_0^\infty \mathcal{I}(\mu_s) ds, \quad \text{ for } \, \rho_t = \int_0^t \mu_s ds,
\end{equation*}
where $\mathcal{I} : \mathcal{P}(E) \to [0,\infty]$ is the rate function appearing in the Donsker-Varadhan results,
\begin{equation*}
\mathcal{I}(\mu) = -\inf_{u \in \mathcal{D}^{++}(L)} \int_E\frac{Lu}{u} d\mu.
\end{equation*}
Then $\eta_t $ satisfies a large deviations principle in $\mathcal{P}(E)$ with rate function $\mathcal{I}$.
\end{proposition}
\begin{proof}
	Because $\{\eta^\tau_t\}_{t > 0}$ satisfies a large deviations principle on $\mathcal{L}(E)$ and the projection $\pi_1 : \mathcal{L}(E) \to \mathcal{P}(E)$ given by 
	\begin{align*}
		\pi _1 (\rho) = \rho _1,
	\end{align*}
		is continuous (Lemma~\ref{lemma:proofs_LDP:continuous_projection}), by the contraction principle the sequence $\{\eta_t ^1\}_{t > 0}$ satisfies a large deviation principle on $\mathcal{P}(E)$ with rate function $\tilde{\mathcal{I}}: \mathcal{P}(E) \to [0,\infty]$ given by
\begin{align*}
\tilde{\mathcal{I}} (\nu) &= \inf\left\{
\mathcal{J}(\rho) = \int_0^\infty \mathcal{I}(\mu_s) \, ds
:
\rho_t = \int_0^t \mu_s \, ds \in \mathcal{L}(E), \rho_1 = \nu
\right\}.
\end{align*} 
It remains to show that $\tilde{\calI} (\nu) = \calI  (\nu)$ for every $\nu \in \calP (E)$. First, in the integral defining $\calJ (\rho)$, the integrand is always positive after time $t=1$. It is therefore enough to consider only integrating to time $t=1$ in the infimum, as we are free to chose the the form of $\rho$ after that time. Thus,
\begin{align*}
	\tilde{\calI} (\nu) = \inf \left\{ \int _0 ^1 \calI  (\mu _t) dt : \rho _t = \int _0 ^t \mu_s ds \in \calL (E), \ \rho_1 = \nu \right\}.
\end{align*}

For a fixed $\nu \in \calP (E)$, take any $\rho_t = \int _0 ^t \mu _s ds \in \calL (E)$ such that $\rho _1 = \nu$. The rate function $\calI $ is convex on $\calP (E)$ and by Jensen's inequality we have
\begin{align*}
	\calI  (\nu) &= \calI  \left( \rho _1 \right) \\
	&=  \calI  \left( \int _0 ^1 \mu _s ds \right) \\
	&\leq \int _0 ^1  \calI  \left( \mu _s \right)ds.
\end{align*}
Taking the infimum over all such $\rho \in \calL(E)$ yields the inequality
\begin{align*}
	\calI  (\nu) \leq \tilde{\calI} (\nu),
\end{align*}
The constant path $\mu _s = \nu$ gives equality and we have that $\tilde{\calI} = \calI $ as functionals on $\calP(E)$.
\end{proof}
\begin{lemma}[Projection is continuous]\label{lemma:proofs_LDP:continuous_projection}
Let $\mathcal{L}(E)$ be the above space with the topology of weak convergence on bounded time intervals. Let $\mathcal{P}(E)$ be equipped with the weak topology. Then the projection $\pi_1 : \mathcal{L}(E) \to \mathcal{P}(E)$ defined by $\pi_1 (\rho) := \rho_1$ is a continuous map.
\end{lemma}
\begin{proof}[Proof of Lemma~\ref{lemma:proofs_LDP:continuous_projection}]
Let $\rho^n \to \rho$ in $\mathcal{L}(E)$. We need to prove that for any bounded and continuous function $g$ on $E$, we have
\begin{equation*}
\int_E g(u) d\rho^n_1(u) \to \int_E g(u) d\rho_1.
\end{equation*}
Since
\begin{equation*}
\int_E g(u) d\rho^n_1(u) = \int_{E \times [0,1]} g(u) d\rho^n(u,s),
\end{equation*}
and $\varphi(u,s) = g(u)$ is continuous and bounded on $E \times [0,\infty)$, this is implied by $\rho^n \to \rho$.
\end{proof}
Armed with Proposition \ref{prop:LDP_strategy}, one way to prove Theorems \ref{thm:proving_LDP:LDP_compact} and \ref{thm:proving_LDP:LDP_non_compact} is to prove the large deviations principle for the empirical measures of the associated sped-up versions of underlying processes and apply the proposition. This is the approach we take and we rely on the following result from \cite{FengKurtz06} for proving the large deviations principles on $\calL (E)$.
\begin{lemma}[Theorem~12.7 of~\cite{FengKurtz06}]
\label{lemma:FengKurtz}
Suppose that the following conditions hold:
\begin{enumerate}[label=(FK.\arabic*)]
	\item The martingale problem for $L$ is well-posed. \label{cond:FK1}
	\item The semigroup $S$ is Feller-continuous. \label{cond:FK2}
	\item The semigroup $S$ is $buc$-continuous. \label{cond:FK3}
	\item There is an index set $\calQ$ and a family of subsets of $E$, $\{ \tilde{K} ^q _n \subset E : q \in \calQ\}$, such that for $q_1, q_2 \in \calQ$, there exists $q_3 \in \calQ$ with $\tilde{K} ^{q_1} _n \cup \tilde{K} ^{q_2} _n \subset \tilde{K} ^{q_3} _n$, and for every $y \in E$, there exists $q \in \calQ$ such that $\lim _{n\to \infty} d (y, \tilde{K} ^{q} _n) =0$. Moreover, for each $q \in \calQ$, $T>0$ and $a >0$, there exists a $\hat q (q,a,T) \in \calQ$ satisfying 
	\begin{align*}
		\limsup _{n \to \infty} \sup _{y \in \tilde{K} ^q _n} \frac{1}{n} \log \bP_{y} \left( Y_t \notin \tilde{K} ^{\hat q (q,a,T)} _n, \ \textrm{some } t \leq nT \right) \leq -a.
	\end{align*}
	\label{cond:FK4}
	\item There exists an upper semicontinuous function $\Psi$ on $E$, $\{ \varphi _n \} \subset \calD ^{++} (B_0)$, and $q_0 \in \calQ$ such that $\Psi$ is bounded above, $\{ y\in E: \Psi(y) \geq c \}$ is compact for each $c \in \bR$, $0 < \inf _{y \in K _n ^{q_0}} \varphi _n (y) < 2 \inf _{y\in E} \varphi _n(y)$, $\inf_{n, y \in E} \varphi_n (y) >0$,
	\begin{align*}
		\lim _{n\to \infty} \frac{1}{n} \log \norm{\varphi _n} =0, \ \ \sup_{n,y} \frac{L \varphi _n (y) }{ \varphi_n (y)} < \infty,
	\end{align*} 
	and for each $q \in \calQ$,
	\begin{align*}
		\lim _{n \to \infty} \sup _{y \in \tilde{K} ^q _n} \left( \frac{L \varphi _n (y) }{ \varphi_n (y)} - \Psi (y) \right) \leq 0, \ \ q\in \calQ.
	\end{align*}
	In addition, for each $n$ and $\beta \in (-\infty, 1]$,
	\begin{align*}
		\lim _{t\to 0} \norm{S(t) \varphi_n ^{\b} - \varphi _n ^\b} =0.
	\end{align*}
	\label{cond:FK5}
	\item For each $a >0$ there exists compact $K$ and $q \in \calQ$ such that
	\begin{align*}
		\limsup _{n \to \infty} \frac{1}{n} \log \bP \left( Y^n _0 \notin K \cap \tilde{K} ^{q} _n  \right) \leq -a.
	\end{align*}
	\label{cond:FK6}
	\item Take $\calC \subset C_b (E)$ separating and define, with $\Psi$ as in \ref{cond:FK5}, 
	\begin{align*}
	H_1 ^{\beta, \Psi}& = \inf_{0 < \kappa \leq 1} \inf_{f \in \mathcal{D}^{++}(L)} \sup_{y \in E} \left[ \beta (y) \cdot p + (1-\kappa) \frac{Lf(y)}{f(y)} + \kappa \Psi(y)\right],\\
	H_2 ^{\beta, \Psi} &= \sup_{\kappa > 0} \sup_{f \in \mathcal{D}^{++}(L)} \inf_{y \in E} \left[ \beta (y) \cdot p + (1+\kappa) \frac{Lf(y)}{f(y)} - \kappa \Psi(y)\right].
	\end{align*}
	It holds that $H_1 ^{\b, \Psi} \leq H_2 ^{\b, \Psi}$, for $\b \in \calC ^d$, $d=1,2.\dots$.  \label{cond:FK7}
\end{enumerate}
Then $\{ \eta ^{\tau} _n \} _n$ satisfies the large deviations principle in $C_E[0, \infty)$ with rate function
\begin{align*}
\label{eq:rateFK}
	\hat \calJ (\rho) = \int _0 ^\infty  I^{\Psi} (\rho_s) ds, \ \rho \in \calL (E), 
\end{align*}
where 
\begin{align*}
   I^{\Psi} (\mu) = -\min\left[\inf _{u \in \calD^{++} (L)}  \int _{E} \frac{Lu}{u} d\mu , \int _E \Psi d\mu\right]. 
\end{align*}
\end{lemma}
\subsubsection{Proof of Theorem~\ref{thm:proving_LDP:LDP_compact}---compact state-space $E$}
\label{section:proof-of-LDP:compact}
As outlined in the previous section, we can prove Theorem \ref{thm:proving_LDP:LDP_compact} by first verifying the conditions of Lemma \ref{lemma:FengKurtz} under the given assumptions and then apply Proposition \ref{prop:LDP_strategy}.
\begin{proof}[Proof of Theorem \ref{thm:proving_LDP:LDP_compact}]
First, Conditions \ref{cond:FK1}-\ref{cond:FK3} follow from the assumption of Feller continuity \ref{item:thm_LDP_compact:Feller} and tightness; see e.g.\ Remark 11.22 in \cite{FengKurtz06}.


Next, Conditions \ref{cond:FK4} and \ref{cond:FK5} always hold for compact $E$: take $\varphi_n \equiv 1$, $\Psi \equiv 0$, $Q = \{q\}$ (singleton), and $K^q_n :=E$ for every $n \in \mathbb{N}$.  For this choice both conditions are met---it is only for non-compact spaces $E$ that these conditions become non-trivial (see the proof of Theorem \ref{thm:proving_LDP:LDP_non_compact}). Condition \ref{cond:FK6} is trivially true for compact $E$.

Remains to verify the inequality $H_1 ^{\beta} \leq H_2 ^{\beta}$. Take $d \geq 1$ and $\beta \in \calC ^d$. With the choice $\Psi \equiv 0$ the definitions of $H_i ^{\beta}: \bR ^d \to \bR$, $i=1,2$, become
	\begin{align*}
	H_1 ^{\beta}& =  \inf_{f \in \mathcal{D}^{++}(L)} \sup_{y \in E} \left[ \beta (y) \cdot p + \frac{Lf(y)}{f(y)} \right],\\
	H_2 ^{\beta} &= \sup_{f \in \mathcal{D}^{++}(L)} \inf_{y \in E} \left[ \beta (y) \cdot p + \frac{Lf(y)}{f(y)} \right].
	\end{align*}
We now show that the required inequality follows from Assumption \ref{item:thm_LDP_compact:principal_eigenvalue}, solvability of the principal eigenvalue problem. 

For any $\beta \in \calC ^d$ and $p \in \bR ^d$, define the map $V_p (y): E \to \bR$ as
\begin{align*}
	V_p (y) = \beta (y) \cdot p
\end{align*}
This is a continuous function on $E$ and for every $p$ there exists a function $f_p \in \calD ^+ (L)$ and real eigenvalue $\lambda_p$ such that
\begin{align*}
	(L + \beta \cdot p) f_p = \lambda _p f_p.
\end{align*}
It follows that, for any $p \in \mathbb{R}^d$, we have 
	\begin{align*}
		\lambda _p = \sup_{y \in E}
	\left[
	\frac{Lf_p(y)}{f_p(y)} + \beta(y) \cdot p
	\right] = \inf_{y \in E}
	\left[
	\frac{Lf_p(y)}{f_p(y)} + \beta(y) \cdot p
	\right],
	\end{align*}
	which leads to the upper bound
	\begin{align*}
	H^\beta_1(p) &=
	\inf_{f \in \mathcal{D}^{+}(L)} \sup_{y \in E} \left[
	\frac{Lf(y)}{f(y)} + \beta(y) \cdot p
	\right]	\\ &\leq 
	\sup_{y \in E}
	\left[
	\frac{Lf_p(y)}{f_p(y)} + \beta(y) \cdot p
	\right]	\\ &=
	\inf_{y \in E}
	\left[
	\frac{Lf_p(y)}{f_p(y)} + \beta(y) \cdot p
	\right]	\\	&\leq
	\sup_{f \in \mathcal{D}^{+}(L)} \inf_{y \in E} \left[
	\frac{Lf(y)}{f(y)} + \beta(y) \cdot p
	\right] \\ &= H^\beta_2(p).
	\end{align*}
This shows that Condition \ref{cond:FK7} of Lemma \ref{lemma:FengKurtz} follows from \ref{item:thm_LDP_compact:principal_eigenvalue}. 
As a result, in the setting of compact $E$, Assumptions  \ref{item:thm_LDP_compact:Feller} - \ref{item:thm_LDP_compact:principal_eigenvalue} ensure that Lemma \ref{lemma:FengKurtz} is applicable. This gives the large deviations principle for the empirical measures associated with sped-up versions of the process $Y$ and Proposition \ref{prop:LDP_strategy} transfers this to the empirical measures of the original process. This concludes the proof of the large deviations principle. The form of the rate function is trivially seen to be equal to the prescribed form because of the choice of $\Psi \equiv 0$.
\end{proof}
\subsubsection{Proof of Theorem~\ref{thm:proving_LDP:LDP_non_compact}---non-compact state space $E$}
\label{section:proof-of-LDP:non-compact}

As in the proof of Theorem~\ref{thm:proving_LDP:LDP_compact}, we prove large deviations of the family of measures $\{\eta_t ^\tau\}_{t > 0}$ introduced at the beginning of Section \ref{sec:aux} by verifying the assumptions of Lemma \ref{lemma:FengKurtz}. Proposition \ref{prop:LDP_strategy} then implies the large deviations principle of the empirical measures $\{\eta_t\}_{t > 0}$ with the prescribed rate function. Whereas the conditions of Lemma \ref{lemma:FengKurtz} where straightforward to verify in the compact setting of Theorem \ref{thm:proving_LDP:LDP_compact}, the non-compact case requires more work. Specifically, because we can no longer assume that there is a solution to the principal eigenvalue problem - such an assumption would not allow us to prove the large deviations principle for the zig-zag process - and the state space is no longer compact, \ref{cond:FK4}-\ref{cond:FK7} are more difficult to verify. A crucial component of the proof of Theorem \ref{thm:proving_LDP:LDP_non_compact} is an inequality that is connected to the necessary comparison principle. To streamline the proof we now state this inequality as a separate result.

For any $V \in C_b(E)$ and $\Psi : E \to \bR$, define $H_1 ^{\Psi}, H_2 ^{\Psi} \in \bR$ by
\begin{equation}
	\begin{split}
	H_1 ^{\Psi} &= \inf_{0 < \kappa \leq 1} \inf_{f \in \mathcal{D}^{++}(L)} \sup_{y \in E} \left[ V(y) + (1-\kappa) \frac{Lf(y)}{f(y)} + \kappa \Psi(y)\right], \label{eq:H1_H2} \\
	H_2 ^{\Psi} &= \sup_{\kappa > 0} \sup_{f \in \mathcal{D}^{++}(L)} \inf_{y \in E} \left[ V(y) + (1+\kappa) \frac{Lf(y)}{f(y)} - \kappa \Psi(y)\right].
\end{split}
\end{equation}
\begin{proposition}\label{prop:non_compact:H1_leq_H2}
	Take any $V \in C_b (E)$ and suppose~\ref{item:thm_LDP_non_compact:mixing} holds and that for any $c\in\mathbb{R}$, the superlevel-set $\{\Psi \geq c\}$ is compact. Then
	\begin{align}
	\label{eq:ineqH1H2}
		H_1 ^{\Psi} \leq H _2 ^{\Psi}.
	\end{align}
\end{proposition}
We first complete the proof of Theorem \ref{thm:proving_LDP:LDP_non_compact}.
\begin{proof}[Proof of Theorem \ref{thm:proving_LDP:LDP_non_compact}]
The proof amounts to showing that Conditions \ref{cond:FK1}-\ref{cond:FK7} of Lemma \ref{lemma:FengKurtz} hold. We start with the ones that are straightforward to obtain from the assumptions of the theorem. 

Conditions \ref{cond:FK1}-\ref{cond:FK3} follow from \ref{item:thm_LDP_compact:Feller} and \ref{item:thm_LDP_compact:tight}. For condition \ref{cond:FK6} the existence of such a compact set follows immediately from the assumption that the initial value $Y(0)$ belongs to a compact set $K \subseteq E$.

We now show that Conditions \ref{cond:FK4} and \ref{cond:FK5} follow from \ref{item:thm_LDP_non_compact:Lyapunov}, the existence of Lyapunov functions $g_1$ and $g_2$ with certain growth properties. We start with \ref{cond:FK4} and define the family of compact sets $K^q_n \subseteq E$ by
\begin{equation*}
K^q_n = \{y \in E\,:\, g_2(y) \leq qn\},\quad q,n \in \mathbb{N}
\end{equation*}
For any $q_1, q_2$ and with $q_3 = \max (q_1, q_2)$, it then holds that
\begin{align*}
	K^{q_1}_n \cup K^{q_2}_n \subseteq K^{q_3}_n, \ \ \forall n \in \bN.
\end{align*}
Because $g_2 (y)$ is finite for any $y\in E$, there exists $q, N \in \bN$ such that $n \geq N$ implies that $y \in K_n ^{q}$. In particular, $\textrm{dist} (y, K_n ^q) = 0$. For the last part of Condition \ref{cond:FK4}, take $q \in \bN$ and $T, a > 0$. It remains to find a $\tilde q$ such that
\begin{align*}
	\limsup _{n \to \infty} \sup _{y \in K  ^{q} _n} \frac{1}{n} \log \bP _y \left( Y_t \notin K ^{\tilde q} _n, \ \textrm{some } t \leq nT \right) \leq -a.
\end{align*}
By Lemma 4.20 in \cite{FengKurtz06}, for any open neighbourhood $\calO$ of $K_n ^q$,
\begin{align}
\label{eq:openN}
	\mathbb{P}\left(Y_t \notin \mathcal{O}, \,\mathrm{some}\, t \leq nT \,|\, Y_0 \in K^q_n\right) \leq \mathbb{P}\left(Y_0 \in K^q_n\right) e^{-\beta _q + nT\gamma(\mathcal{O})},
\end{align}
where the constants $\beta_q$ and $\gamma(\mathcal{O})$ are given by
\begin{align*}
\beta_q = \inf_{E \setminus \mathcal{O}} g_2 - \sup_{K^q_n} g_2,
\end{align*}
and
\begin{align*}
\gamma(\mathcal{O}) = \max \left(\sup_{\mathcal{O}}e^{-g_2} B e^{g_2}, 0\right).
\end{align*}
By the growth condition for $g_2$ (part (a) of \ref{item:thm_LDP_non_compact:Lyapunov}) for any $\tilde q > \hat q > q$ large enough, there exists an open set $\calO$ such that
\begin{align*}
	K_n ^q \subseteq K_n ^{\hat q} \subseteq \calO \subseteq K^{\tilde q} _n.
\end{align*} 
By definition, $g_2 \leq nq$ on $K ^q _n$ and because $K ^{\hat q} _n \subseteq \calO$, we have $g_2 \geq \hat{q} n$ on $E \setminus \calO$. Combined with the upper bound $\gamma (\calO) \leq \gamma (E)$ this gives, starting from \eqref{eq:openN},
\begin{align*}
	\frac{1}{n} \log \bP \left( Y_t \notin K ^{\tilde q} _n, \ \textrm{some } t\leq nT | Y_0 \in K_n ^q \right) &\leq T \gamma(\calO) - \frac{1}{n} \beta _q \\
	&\leq T \gamma(E) +q - \hat{q}.
\end{align*}
The last part of Condition \ref{cond:FK4} is now straightforward to obtain. First, take $\hat q = \hat q (q, a, T)$ large enough that the right-hand side of the last display is bounded by $-a$:
\begin{align*}
	T \gamma(E) +q - \hat{q} \leq -a.
\end{align*}
Next, choose $\tilde q = \tilde q (q, a, T)$ large enough that there is an open set $\calO$ such that $K ^{\hat q} _n \subseteq \calO \subseteq K^{\tilde q} _n$. The asymptotic statement then follows, which concludes the verification of condition \ref{cond:FK4} of Lemma \ref{lemma:FengKurtz}.
\smallskip

To show that condition \ref{cond:FK5} is fulfilled we generalize the arguments used in Example 11.24 in \cite{FengKurtz06}. The functions $\varphi_n$ are constructed from the Lyapunov functions $g_1$ and $g_2$. First, define
\begin{equation}\label{eq:proving_LDP:r_n}
r_n := \sup \left\{g_1(y) \,: \, y \in E,\, g_1(y) g_2(y) \leq n^2\right\}.
\end{equation}
Then $r_n \to \infty$ and $r_n/n \to 0$ as $n \to \infty$, since~\ref{item:thm_LDP_non_compact:Lyapunov2} and the condition in the set imply
\begin{equation*}
\frac{r_n}{n} = \frac{g_1(y_n)}{n} \leq \sqrt{\frac{g_1(y_n)}{g_2(y_n)}} \to 0.
\end{equation*}
Furthermore, for each $q$ there exists $n_q$ such that $n \geq n_q$ implies
\begin{equation*}
K^q_n \subseteq \left\{y\,:\, g_1(y) \leq r_n\right\}.
\end{equation*}
For a smooth, non-decreasing and concave function $\rho : [0,\infty) \to [0,2]$ satisfying $\rho(r) = r$ for $0 \leq r \leq 1$ and $\rho(r) = 2$ for $r \geq 3$, define the functions $\varphi_n$ by cutting off $g_1$:
\begin{equation}\label{eq:proving_LDP:varphi_n}
\varphi_n(y) := e^{r_n} \rho(e^{-r_n} e^{g_1(y)}).
\end{equation}
We have $\varphi_n = e^{g_1}$ on the compact sets $K^q_n$. Setting $\Psi = e^{-g_1} B e^{g_1}$, we therefore obtain
\begin{equation*}
\frac{L \varphi_n(y)}{\varphi_n(y)} = \Psi(y),\quad y \in K^q_n,\, n \geq n_q.
\end{equation*}
The fact that $r_n/n\to 0$ as $n\to\infty$ implies $n^{-1}\log\|\varphi_n\| \to 0$ as $n\to\infty$. For proving that
\begin{equation*}
    \sup_{n,y}\frac{L\varphi_n(y)}{\varphi_n(y)}<\infty,
\end{equation*}
it is sufficient to show that for any positive function $u:E \to (0,\infty)$ in the domain of $B$ and for any $y_0 \in E$, we have
\begin{equation}\label{eq:proof-LDP-noncompact:ineq-pos-max-pr}
\frac{B (\rho(u))(y_0)}{\rho(u)(y_0)} \leq \frac{\max\left(B u(y_0),0\right)}{u(y_0)}.
\end{equation}
Then with $u=e^{-r_n}e^{g_1}$ and noting that $L\varphi_n=B\varphi_n$, by linearity we obtain
\begin{equation*}
    \frac{L\varphi_n}{\varphi_n} \leq \frac{\max\left(B e^{g_1},0\right)}{e^{g_1}},
\end{equation*}
and the result follows since $\Psi(y)=e^{-g_1(y)} \left(B e^{g_1}\right)(y)\to-\infty$ as $|y|\to\infty$. Hence we are left with verifying~\eqref{eq:proof-LDP-noncompact:ineq-pos-max-pr}.

If $u(y_0) \in (3,\infty)$, then $\rho(u)(y_0)$ is maximal. Hence by the positive maximum principle, $B\rho(u)(y_0) \leq 0$, and the inequality follows. If $u(y_0) \in (0,3]$, then $a_0 := \rho'(u(y_0)) \in [0,1]$, the region where $\rho$ goes from slope one to slope zero. Consider the function $f_0 := \rho(u) - a_0 u$. Since $g_0(s) := \rho(s) - a_0 s$ is maximal for $s_0$ satisfying $\rho'(s_0) = a_0$, we obtain that $y_0$ is an optimizer, that is $f_0(y_0) = \sup_y f(y)$. Furthermore, $g_0(s_0) \geq 0$, so by the positive maximum principle $Bf_0(y_0) \leq 0$. By linearity of $B$ and since $r \leq \max(r,0)$, we obtain the inequality $B \rho(u)(y_0) \leq a_0 \cdot \max\left(Bu(y_0),0\right)$. Hence
\begin{equation*}
\frac{B \rho(u)(y_0)}{\rho(u)(y_0)} \leq \frac{a_0}{\rho(u)(y_0)}\max(Bu(y_0),0) \leq \frac{1}{u(y_0)} \max(Bu(y_0),0),
\end{equation*}
using that $0 \leq g_0(s_0) = \rho(u)(y_0)-a_0 u(y_0)$. This finishes the verification of~\eqref{eq:proof-LDP-noncompact:ineq-pos-max-pr}.
\smallskip

It remains to show that condition \ref{cond:FK7} is fulfilled. However, this is precisely the conclusion of Proposition \ref{prop:non_compact:H1_leq_H2} - the function $V(y) = \beta (y)\cdot p$ where $\beta$ is as in condition (vii) is an element of $C_b (E)$, and by~\ref{item:thm_LDP_non_compact:Lyapunov3}, the function $\Psi$ has compact superlevel-sets.
\smallskip

We have shown that under the assumptions of the theorem, all conditions of Lemma \ref{lemma:FengKurtz} are fulfilled. The large deviations principle for the empirical measures of the sped-up versions thus holds and Proposition \ref{prop:LDP_strategy} then gives the large deviations principle for the empirical measures $\{ \eta _t \}$ associated with $Y$. 

We are left with showing that the rate function $I ^{\Psi}$ of Proposition \ref{prop:LDP_strategy} satisfies
\begin{align*}
	I ^{\Psi} (\mu) = -\inf _{u \in \calD^{++} (L)}  \int _{E} \frac{Lu}{u} d\mu.
\end{align*}
Below, we prove that
\begin{equation*}
    \limsup_{n\to\infty}\int_{E}\frac{L\varphi_n}{\varphi_n} d\mu \leq \int_E \Psi d\mu.
\end{equation*}
Then 
\begin{equation*}
    \inf_{u \in \calD^{++} (L)}\int_E\frac{Lu}{u} d\mu \leq \inf_n \int_E \frac{L\varphi_n}{\varphi_n} d\mu \leq \int_E \Psi d\mu,
\end{equation*}
and hence the rate function is given by
\begin{align*}
	I^{\Psi} (\mu) = -\min \left[ \inf _{u \in \calD^{++} (L)}  \int _{E} \frac{Lu}{u} d\mu,  \int _E \Psi d\mu \right] = -\inf _{u \in \calD^{++} (L)}  \int _{E} \frac{Lu}{u} d\mu.
\end{align*}
To see that the functions $\varphi_n$ satisfy the limsup inequality, note that $\Psi$ has compact super-level sets and $\Psi(y)\to-\infty$ as $|y|\to\infty$. Since the compact sets $K_n^q$ exhaust $E$ in the sense that $E = \cup_n K^q_n$ and $K^q_n\subseteq K^q_{n+1}$, there exists a constant $C > 0$ such that
\begin{equation*}
    f_n = -\frac{L\varphi_n}{\varphi_n} + C \geq 0 .
\end{equation*}
Pointwise, we have $f=-\Psi + C = \liminf_n f_n$. Therefore, by Fatou's lemma
\begin{equation*}
    \liminf_{n\to\infty} \int_E \left[-\frac{L\varphi_n}{\varphi_n} + C\right]\,d\mu \geq \int_E\left[-\Psi + C\right]\,d\mu,
\end{equation*}
and the required limsup inequality follows from reorganizing.
\end{proof}

We now prove the important Proposition \ref{prop:non_compact:H1_leq_H2}. The proof is essentially a combination of different arguments from Chapter 11 and Appendix B of \cite{FengKurtz06} (see especially Lemmas 11.12, 11.37, B.9-B.11 for full details). We present the proof as to make the presentation self-contained and give a succinct derivation of the results for the setting considered in this paper. The main difference compared to the arguments in \cite{FengKurtz06} is that we work with measures $\nu \in \calP _c (E) $ rather than imposing the condition $\int _E \Psi \nu > -\infty$, and we must verify that we can indeed modify the latter.

\begin{proof}[Proof of Proposition \ref{prop:non_compact:H1_leq_H2}]
The strategy is to find two constants, depending on $V$, $c_V^\ast$ and $c_V ^{\ast\ast}$ such that $c_V ^\ast \geq c_V ^{\ast \ast} $ and
\begin{align}
\label{eq:ineqH}
	H_1 ^{\Psi} \leq c_V ^{\ast\ast}, \, \, \textrm{and }\,  H_2 ^{\psi} \geq c_V ^\ast,
\end{align} 
To achieve this we study the following quantity: for $\nu \in \calP_c (E)$, define
\begin{align*}
	c_V (\nu) = \limsup _{t \to \infty} \frac{1}{t} \log \bE \left[ \mathrm{exp}\left\{ \int _0 ^t V(Y(s))ds \right\}  \right].
\end{align*}
It can be shown - see e.g. Lemma B.9 in \cite{FengKurtz06} - that under \ref{item:thm_LDP_non_compact:mixing}, $c_V (\nu)$ exists for each $\nu \in \calP _c (E)$ and the necessary inequalities for $H_i ^{\Psi}$ can be derived for
\begin{align*}
	c_V^{\ast} = \inf_{\nu \in \calP _c (E)} c_V(\nu) \quad \text{ and } \quad
	c_V^{\ast\ast} = \sup_{\nu \in \calP _c (E)} c_V(\nu).
\end{align*}
Cleary $c_V ^\ast \leq c_V ^{\ast \ast}$. However it can be shown, again using~\ref{item:thm_LDP_non_compact:mixing}, that the two quantities are in fact equal, that $c_V (\nu)$ is independent of $\nu$ on $\calP _c (E)$. If we can prove \eqref{eq:ineqH} this would then yield the claim.

We start with the upper bound
\begin{align*}
	H_1 ^{\Psi} \leq c_V ^{\ast\ast}.
\end{align*}
An argument similar to what will follow is also used in \cite{DonskerVaradhan75}, in the proof of their Lemma~2.

Because $\Psi$ has compact superlevel-sets $\{ \Psi \geq c \}$, $c\in \bR$, and $\Psi (y) \to - \infty$ as $\norm{y} \to \infty$, it can be shown using the arguments of Lemma B.11 of \cite{FengKurtz06} that
\begin{align*}
	\sup_{\nu \in \mathcal{P}_c(E)} \inf_{f \in \mathcal{D}^{++}(L)} \int_E \left(V + \frac{Lf}{f}\right)d\nu \leq c_V^{\ast\ast}.
\end{align*}
It therefore suffices to show that 
\begin{align}
\label{eq:ineqH1}
H_1 ^{\Psi} \leq \sup_{\nu \in \mathcal{P}_c(E)} \inf_{f \in \mathcal{D}^{++}(L)} \int_E \left(V + \frac{Lf}{f}\right)d\nu.
\end{align}
For any finite collection of functions $f_1, \dots , f_m$ in $\calD ^{++} (L)$ and scalars $\alpha_i \geq 0$, $i=1,\dots , m$, $\sum \alpha_i =1$, we have 
\begin{align*}
	H_1 ^{\Psi} \leq \inf_{0 < \kappa \leq 1} \inf_{\alpha_i} \inf_{f_1,\dots,f_m} \sup_{y \in E} \left[
	V(y) + (1-\kappa)\sum_{i=1}^m \alpha_i \frac{Lf_i(y)}{f_i(y)} + \kappa \Psi(y).
	\right]
\end{align*}
Define
\begin{align*}
	h_t = \frac{1}{t} \int _0 ^t S(\tau) \prod _{i=1} ^m f_i ^{\alpha _i} d \tau. 
\end{align*}
Then
\begin{align*}
\lim_{t \to 0} h_t = \prod_i f_i^{\alpha_i},
\end{align*}
and we have the upper bound
\begin{align*}
 \lim_{t \to 0}Lh_t \leq \prod_i f_i^{\alpha_i} \sum_i \alpha_i \frac{Lf_i}{f_i},
\end{align*}
where the convergence is uniform.

We can select a sequence of functions $\{ f_i \} $ from $\calD ^{++} (L)$ such that for any $\mu \in \calP (E)$,
	\begin{equation*}
	\inf_{i} \int_E \frac{Lf_i}{f_i} d\mu = \inf_{f \in \mathcal{D}^{++}(L)} \int_E \frac{Lf}{f} d\mu.
	\end{equation*}
	Specialising to these functions, for any $m \in \bN$ and $\kappa > 0$ we have the upper bound
	\begin{equation*}
	H_1 ^{\Psi} \leq \inf_{\alpha_i} \sup_{y \in E} \left[
	V(y) + (1-\kappa)\sum_{i=1}^m \alpha_i \frac{Lf_i(y)}{f_i(y)} + \kappa \Psi(y).
	\right]
	\end{equation*}
	The functions $V + Lf_i/f_i$ are bounded, but a priori there is no guarantee that the supremum is attained in a given compact set. However, because $\Psi(y) \to -\infty$ as $\norm{y} \to \infty$, for any $m \in \bN$ and $\kappa > 0$, there exists a constant $\ell = \ell(m,\kappa) > 0$ such that the supremum is attained in the compact set $K_{\ell} = \{\Psi \geq -\ell\}$. Therefore, if we define $\mathcal{K}_\ell = \{\nu \in \mathcal{P}(E): \nu(K_\ell) = 1\}$, then
	\begin{align*}
	H_1 ^{\Psi} &\leq \inf_{\alpha_i} \sup_{y \in K_{\ell}} \left[
	V(y) + (1-\kappa)\sum_{i=1}^m \alpha_i \frac{Lf_i(y)}{f_i(y)} + \kappa \Psi(y)
	\right]	\\
	&= \inf_{\alpha_i} \sup_{\nu \in \mathcal{K}_\ell} \left[
	\int_E\left(V(y) + (1-\kappa)\sum_{i=1}^m \alpha_i \frac{Lf_i(y)}{f_i(y)} + \kappa \Psi(y) \right) d\nu(y)
	\right].
	\end{align*}
	For any $\ell$, we have $\mathcal{K}_\ell \subseteq \mathcal{P}_c(E)$, so that
	\begin{equation*}
	H_1 ^{\Psi} \leq \inf_{\alpha_i} \sup_{\nu \in \mathcal{P}_c(E)} \left[
	\int_E\left(V + (1-\kappa)\sum_{i=1}^m \alpha_i \frac{Lf_i}{f_i} + \kappa \Psi \right) d\nu
	\right].
	\end{equation*}
	For any $m \in \bN$, the set $\{ \alpha_i: \alpha _i \geq 0, \sum _{i=1} ^m \alpha _i =1 \}$ is compact and the infimum and supremum in the last display can be exchanged by Sion's Theorem. This yields
	\begin{align*}
	H_1 ^{\Psi} &\leq  \sup_{\nu \in \mathcal{P}_c(E)} \inf_{\alpha_i} \left[
	\int_E\left(V + (1-\kappa)\sum_{i=1}^m \alpha_i \frac{Lf_i}{f_i} + \kappa \Psi \right) d\nu
	\right]	\\
	&= \sup_{\nu \in \mathcal{P}_c(E)} \left[
	\int_E\left(V + (1-\kappa)\min_{i\leq m} \frac{Lf_i}{f_i} + \kappa \Psi \right) d\nu
	\right],
	\end{align*}
	where we have used that $\inf_{\alpha_i}\sum \alpha_i x_i = \min_i x_i$ for non-negative $x$ and $f_i \in \calD ^{++} (L)$. Taking the infimum over $\kappa$ and the limit $m \to \infty$,
	\begin{align*}
	H_1 ^{\Psi} &\leq \lim_{m \to \infty} \inf_{0 < \kappa \leq 1} \sup_{\nu \in \mathcal{P}_c(E)} \left[
	\int_E V d\nu + (1-\kappa) \min_{i \leq m} \int_E \frac{Lf_i}{f_i} d\nu + \kappa \int_E \Psi d\nu
	\right]	\\
	& = \lim_{m \to \infty} \sup_{\nu \in \mathcal{P}_c(E)} \left[
	\int_E V d\nu + \min\left\{\min_{i \leq m} \int_E \frac{Lf_i}{f_i} d\nu, \int_E \Psi d\nu \right\}
	\right].
	\end{align*}
	The limit and supremum can be shown to commute similarly to the last part of the proof of Lemma 11.12 in \cite{FengKurtz06}, leading to
	
	\begin{align*}
	H_1 ^{\Psi} &\leq  \sup_{\nu \in \mathcal{P}_c(E)} \left[
	\int_E V d\nu + \min\left\{\inf_{f \in \mathcal{D}^{++}(L)} \int_E \frac{Lf}{f} d\nu, \int_E \Psi d\nu\right\}
	\right] \\
	&\leq \sup_{\nu \in \mathcal{P}_c(E)} \left[
	\int_E V d\nu + \inf_{f \in \mathcal{D}^{++}(L)} \int_E \frac{Lf}{f} d\nu \right] .
	\end{align*}
%
	This completes the proof of the upper bound for $H^ {\Psi} _1$.
	
	Next, we move to the lower bound for $H_2 ^{\Psi}$. Take $\lambda < c_V^\ast$. We prove that for any $\varepsilon > 0$, we have $H_2 ^{\Psi} \geq \lambda - \varepsilon$. To this end, 
	we define the new semigroup $ \{ T(t) \} $ by
	\begin{align*}
	(T(t)f)(y) = \bE \left[ f(Y_t) e^{\int_0 ^t V(Y_s)ds} | Y(0)=y \right],
	\end{align*}
	set
	\begin{align*}
	    R_{\lambda} ^t f = \int _0 ^t e^{-\lambda s} T(s) g ds,
	\end{align*}
	and take $\Gamma$ to be the collection of functions $f_{\gamma}$ of the form
	\begin{align*}
	    f_{\gamma} = \int _0 ^\infty R_{\lambda} ^t 1 \gamma (dt), \ \ \gamma \in \calP ([0,\infty)).
	\end{align*}
	Then $\Gamma \subseteq \mathcal{D}^{++}(L)$ and for any $f \in \Gamma$ we have the uniform lower bound
	\begin{equation}
	\label{eq:uniformLower}
	V(y) + (1+\kappa) \frac{Lf (y)}{f(y)}\geq
	-(1-2\kappa)\|V\|, \ \ y \in E.
	\end{equation}
    Because $\Gamma \subseteq \calD ^{++} (L)$, for any $\kappa >0$ we have the lower bound
	\begin{equation*}
	H_2 ^{\Psi} \geq \sup_{f \in \Gamma} \inf_{y\in E} \left[ V(y) + (1+\kappa)\frac{Lf(y)}{f(y)} - \kappa \Psi (y) \right].
	\end{equation*}
	Due to the uniform lower bound \eqref{eq:uniformLower} and the fact that $\Psi(y) \to -\infty$ as $\norm{y} \to \infty$, for any $\kappa$ there exists an $\ell = \ell(\kappa)$ such that the infimum over $E$ is attained in the compact set $K_\ell = \{ \Psi \geq -\ell\}$. Therefore,

	\begin{align*}
	H_2 ^{\Psi} &\geq \sup_{f \in \Gamma} \inf_{y\in K_\ell} \left[ V(y) + (1+\kappa)\frac{Lf(y)}{f(y)} - \kappa \Psi(y) \right] \\
	&= \sup_{f \in \Gamma} \inf_{\nu \in \mathcal{K}_\ell} \left[ \int_E \left( V + (1+\kappa)\frac{Lf}{f} - \kappa \Psi\right)\,d\nu \right]	\\
	&= \sup_{f \in \Gamma} \inf_{\nu \in \mathcal{K}_\ell} \frac{1}{\int_E f d\nu} \left[ -\kappa \int_E (V + \Psi) f d\nu  + \int_E (1+\kappa) (V+L)f d\nu \right],
	\end{align*}
	where $\mathcal{K}_\ell = \{\nu \in \mathcal{P}(E): \nu(K_\ell) = 1\}$. The second equality follows from the fact that $\inf_\nu \int_E (a/b) d\nu = \inf_\nu (\int_E a d\nu)/ (\int_E b d\nu)$ for $b > 0$. By compactness of $\mathcal{K}_\ell$ and the fact that both $\mathcal{K}_\ell$ and $\Gamma$ are convex, the infimum and supremum are exchangable by Sion's Theorem. This gives the lower bound
	\begin{align*}
	H_2 ^{\Psi}	&\geq \inf_{\nu \in \mathcal{K}_\ell} \sup_{f \in \Gamma} \frac{1}{\int_E f d\nu} \left[ -\kappa \int_E (V + \Psi) f d\nu  + \int_E (1+\kappa) (V+L)f d\nu \right]	\\
	&\geq \inf_{\nu \in \mathcal{P}_c(E)} \sup_{f \in \Gamma}  \frac{1}{\int_E f d\nu} \left[ -\kappa \int_E (V + \Psi) f d\nu  + \int_E (1+\kappa) (V+L)f d\nu \right],
	\end{align*}
	The second estimate follows since $\mathcal{K}_\ell \subseteq \mathcal{P}_c(E)$ for any $\ell$. The rest of the proof follows arguments similar to those used in~\cite{FengKurtz06}: taking the limit $\kappa \to 0$, and moving it inside the infimum and supremum, we obtain the lower bound
	\begin{align*}
	H_2 ^{\Psi} & \geq \inf_{\nu \in \mathcal{P}_c(E)} \sup_{f \in \Gamma} \left[ \frac{1}{\int_E f d\nu}\int_E (V+L)f d\nu \right].
	\end{align*}
	Therefore, for any $\varepsilon$, there exists a $\nu_\varepsilon \in \mathcal{P}_c(E)$ such that
	\begin{align*}
	H_2 ^{\Psi} \geq \sup_{f \in \Gamma} \left[ \frac{1}{\int_E f d\nu_\varepsilon}\int_E (V+L)f d\nu_\varepsilon \right] - \varepsilon.
	\end{align*}
	%
	There exist functions $f_t \in \Gamma$ satisfying
	\begin{align*}
	\frac{\int_E (V+L) f_t d\nu}{\int_E f_t d\nu} = \lambda + \frac{ \int_E e^{-\lambda t} T(t) 1 d\nu - 1}{\int_E f_t d\nu},
	\end{align*}
	for any $\nu \in \mathcal{P}_c(E)$. Specialising to such $f_t$, we obtain
	\begin{align*}
	H_2 ^{\Psi} \geq \lambda + \frac{\int_E e^{-\lambda t} T(t)1 d\nu_\varepsilon - 1}{\int_E f_t d\nu_\varepsilon} - \varepsilon.
	\end{align*}
	Since $\limsup_{t \to \infty} \int_E e^{-\lambda t} T(t) 1 d\nu_\varepsilon = \infty$, the second term is positive for $t$ large enough, giving the bound
	\begin{equation*}
	H_2 ^{\Psi} \geq \lambda - \varepsilon.
	\end{equation*}
	This completes the proof of the lower bound for $H_2 ^{\Psi}$, and thereby the lemma.
\end{proof}

\subsection{Large deviations proofs for the empirical measure of the zig-zag process}
In this section, we prove the large deviations theorems for the empirical measures of the zig-zag process.
\subsubsection{Proof of Theorem~\ref{thm:LDP-zigzag:compact}---compact case}
\label{section:zig-zag:LDP:compact}
For the proof of Theorem~\ref{thm:LDP-zigzag:compact}, recall that the zig-zag generator takes the form
\begin{equation*}
    Lf(x,v)=v\partial_x f(x,v) + \lambda(x,v)\left[f(x,-v)-f(x,v)\right], \quad (x,v)\in E:= \mathbb{T}\times\{\pm 1\}.
\end{equation*}
It is enough to show that assumptions \ref{item:thm_LDP_compact:Feller}-\ref{item:thm_LDP_compact:principal_eigenvalue} hold for the zig-zag process on $E$, the result then follows form Theorem \ref{thm:proving_LDP:LDP_compact}.


We first verify that $L$ is a closed operator that generates the zig-zag process. Note that $L$ is a restriction of the extended generator (see Section~\ref{sec:zig-zag}).
We verify that $L$ is a closed operator. Let $\{ f_n \}$ be a sequence in $\mathcal{D}(L)$ such that $f_n \to f$ and $Lf_n \to g$, some $f,g$, both uniformly on~$E$. Then
\begin{equation*}
\lim_{n \to \infty} v \partial_x f_n(x,v) = g(x,v) - \lambda(x,v) \left[ f(x,-v) - f(x,v) \right].
\end{equation*}
We can represent $f_n (x,v)$ as 
\begin{equation*} f_n(x,v) = f_n(0,v) + v \int_0^x v \partial_x f_n(\xi,v) \, d\xi,
\end{equation*}
and from the dominated convergence theorem we obtain that
\begin{equation*}
f(x,v) = f(0,v) + v \int_0^x \left[ g(\xi,v) - \lambda(\xi,v) \left( f(\xi,-v) - f(\xi,v) \right) \right] \, d\xi.
\end{equation*}
In particular, $f \in \mathcal{D}(L)$, and $Lf = g$ follows from taking derivative $\partial_x$ and multiplying by $v$. 
	
The Feller-continuity property \ref{item:thm_LDP_compact:Feller} of the zig-zag semigroup $S$ is proven in Proposition~4 of~\cite{bierkens2017piecewise}. Since $\mathbb{T}$ is compact, this also follows from the boundedness of the continuous rates $\lambda$, see \cite[Theorem 27.6]{Davis1993}.

It remains to verify assumption \ref{item:thm_LDP_compact:principal_eigenvalue}, the principal-eigenvalue problem. Take $V \in C(E)$. We will show that for any constant $\gamma > \sup_E V$, as a map from $C(E)$ to $\mathcal{D}(L) \subseteq C(E)$, the resolvent
	\begin{align}
	\label{eq:resolvent}
	R_\gamma = \left(
	\gamma - (V+L)
	\right)^{-1},
	\end{align}
	is compact and strongly positive; here strongly positive means that if $f \geq 0$ and $f \neq 0$, then $R_\gamma f > 0$ on $E$. Given strong positivity and compactness, by the Krein-Rutman theorem there exists a strictly positive function $g \in C(E)$ and a real eigenvalue $\beta >0 $ such that
	\begin{equation*}
	\left(
	\gamma - (V+L)
	\right)^{-1} g = \beta g.
	\end{equation*}
 The resolvent maps into the domain of $L$, so that $g \in \mathcal{D}(L)$. An application of $\gamma - (V+L)$ in the eigenvalue equation gives
	\begin{equation*}
	(V+L) g = \left(\gamma - \frac{1}{\beta}\right) g.
	\end{equation*}
	This is precisely \ref{item:thm_LDP_compact:principal_eigenvalue} with function $g$ and eigenvalue $(\gamma - 1/\beta)$. 
	
	We are left with verifying that the resolvents defined by \eqref{eq:resolvent} are strongly positive and compact. For strong positivity, because $V$ is continuous on $E$, it is sufficient to prove strong positivity of $(\gamma - L)^{-1}$; see  \cite[Proposition C-III-3.3]{arendt1986one}. The resolvent $(\gamma - L)^{-1}$ exists for any $\gamma > 0$, and is given by
	\begin{equation}
	\label{eq:resolvent2}
	(\gamma - L)^{-1} f = \int_0^\infty e^{-\gamma t} S(t)f \, dt.
	\end{equation}
	The semigroup associated to the zig-zag process is irreducible in the following sense: for any $f \in C(E)$ such that $f \geq 0$ and $ f \neq 0$, 
	\begin{align*}
	\cup_{t \geq 0}\left\{
	z \in E \,:\, S(t) f(z) > 0
	\right\} = E.
	\end{align*}
	Combined with \eqref{eq:resolvent2} this implies strong positivity of $(\gamma - L) ^{-1}$; see \cite[Definition C-III-3.1]{arendt1986one}. 
	\smallskip
	
	For compactness of $R_\gamma$, let $A \subseteq C(E)$ be bounded. We show that the image $B := R_\gamma(A) \subseteq C(E)$ is bounded and equi-continuous. Compactness of the resolvent then follows from an application of the Arzelà-Ascoli theorem. To show boundedness, by dissipativity of $L$ we obtain, for any $g \in B$,
	\begin{align*}
	(\gamma -\|V\|_E) \|g \| &\leq
	\|(\gamma - (V+L))g\| \\
	& \leq \sup_{f \in A} \|f\| \\
	&< \infty.
	\end{align*}
	Hence $B$ is bounded by $C_A / (\gamma - \|V\|_E)$, where $ C_A := \sup _{f \in A} \|f\|$, and we end the proof by showing that $B$ is equi-continuous. For any $g \in B$ we have $(\gamma - (V + L))g = f$ for some $f \in A$, which implies
	\begin{equation*}
	v \partial_x g(x,v) = f(x,v) + V(x,v) g(x,v) + \gamma g(x,v) - \lambda(x,v) (g(x,-v) - g(x,v).
	\end{equation*}
	By boundedness of the functions $\lambda(x,v)$ and $V$ on $E$ and the sets $A$ and $B$,
	\begin{equation*}
	\sup_{g \in B} \|\partial_x g\| \leq C \sup_{g \in B}\|g\| + \sup_{f \in A}\|f\| < \infty.
	\end{equation*}
	Hence functions in $B$ have uniformly bounded derivatives, and as a consequence, $B$ is equi-continuous. It follows that $R_{\gamma}$ in \eqref{eq:resolvent} is compact and strongly continuous. This finishes the verification of \ref{item:thm_LDP_compact:principal_eigenvalue} and we have shown that assumptions \ref{item:thm_LDP_compact:Feller}-\ref{item:thm_LDP_compact:principal_eigenvalue} hold for the zig-zag process on the compact state space $\bT \times \{ \pm 1\}$. An application of Theorem \ref{thm:proving_LDP:LDP_compact} then proves the claimed large deviations principle.
\subsubsection{Proof of Theorem~\ref{thm:LDP-zigzag:non-compact-1d}---non-compact case}
\label{section:zig-zag:LDP:non_compact}
For notational simplicity we take $E = \bR \times \{ \pm 1\}$.
Similar to the proof of Theorem \ref{thm:LDP-zigzag:compact}, the strategy is to verify the conditions of the more general large deviations result Theorem \ref{thm:proving_LDP:LDP_non_compact}, which covers the non-compact setting. That is, it suffices to verify \ref{item:thm_LDP_compact:Feller}, \ref{item:thm_LDP_compact:tight}, \ref{item:thm_LDP_non_compact:Lyapunov} and \ref{item:thm_LDP_non_compact:mixing}.
\smallskip

Condition \ref{item:thm_LDP_compact:Feller}, Feller-continuity of the Markov semigroup, is proven in Proposition 4 of \cite{bierkens2017piecewise}. 

Next, we use Theorem 7.2 of \cite{ethier2009markov} to verify \ref{item:thm_LDP_compact:tight}. 
Define the metric $d$ on $E$ as 
	\begin{align*}
	d((x,v),(y,v')) = |x-y|_\mathbb{R} + |v-v'|,
	\end{align*} 
	and for any path $\gamma \in D_E[0,\infty)$ set
	\begin{equation*}
	w'(\gamma,\delta,T) = \inf_{\{t_i\}} \max_i \sup_{s,t \in [t_i,t_{i+1})} d(\gamma(s),\gamma(t)),
	\end{equation*}
	where the infimum is taken over finite partitions $\{ t_i \}$ of $[0,T]$ such that $\min_i |t_{i+1}-t_i| > \delta$. Theorem 7.2 of \cite{ethier2009markov} states that tightness of $\{ \bP _y : \ y \in K \}$ is equivalent to the following two conditions: 
\begin{enumerate}[label =(\arabic*)]
	\item \label{item:proof-LDP-zigzag:tightness1} For any $\varepsilon > 0$ and rational $t > 0$, there exists a compact set $K_{\varepsilon,t} \subseteq E$ such that
	\begin{equation*}
	\inf_{y \in K} \mathbb{P}_y\left[Y_t \in K_{\varepsilon,t} \right] \geq 1 - \varepsilon.
	\end{equation*}
	\item \label{item:proof-LDP-zigzag:tightness2} 
	For any $\varepsilon > 0$ and $T > 0$, there exists a $\delta > 0$ such that
	\begin{equation*}
	\sup_{y \in K} \mathbb{P}_y\left[w'\left(Y,\delta,T\right) \geq \varepsilon\right] \leq \varepsilon.
	\end{equation*}
\end{enumerate}
The spatial component $X_t$ of the zig-zag process propagates with finite speed. This implies that there exists a compact set $K_t \subseteq \mathbb{R}$ such that if $y \in K$, then
\begin{equation*}
\mathbb{P}_y \left[X_t \in K_t \right] = 1.
\end{equation*}
For any $\varepsilon > 0$ and $t>0$, taking $K_{\varepsilon,t} = K_t \times \{\pm 1\}$ gives~\ref{item:proof-LDP-zigzag:tightness1}. 
\smallskip

For part \ref{item:proof-LDP-zigzag:tightness2}, let $\varepsilon > 0$ and $T > 0$. For any realization $Y(\omega)$ of the zig-zag process on the time interval $[0,T]$, if the sojourn times $\tau_i$ satisfy $\min_i \tau_i > 2\delta$, then $w'(Y(\omega),\delta,T) \leq 2\delta$. In particular, for $\delta$ small enough, $w'(Y(\omega),\delta,T) < \varepsilon$. The probability of having at least one sojourn time that is less than $2\delta$ can be estimated uniformly over starting points $y\in K$. Let $K(T)$ denote the set of points that the zig-zag can reach in the time interval $[0,T]$ when starting in the set $K$ and set $\lambda_K = \sup_{y \in K(T)}\lambda(y)$, a uniform upper bound on the jump rates $\lambda(x,v)$. An estimate for the probability of at least one sojourn time that is less than $2\delta$ is then given by
\begin{equation*}
\sup_{y\in K} \mathbb{P}_y\left[\min_i \tau_i \leq 2\delta\right] \leq 1 - e^{-\lambda_K 2\delta}.
\end{equation*}
For any $y\in K$ we obtain the bound
\begin{align*}
\mathbb{P}_y\left[w'(Y,\delta,T) \geq \varepsilon\right] &= \mathbb{P}_y\left[\{w'(Y,\delta,T)\geq \varepsilon\} \cap \{\min_i \tau_i > 2\delta\}\right] \\
& \quad + \mathbb{P}_y\left[\{w'(Y,\delta,T)\geq \varepsilon\} \cap \{\min_i \tau_i \leq 2\delta\}\right]\\
&\leq 0 + \mathbb{P}_y\left[\{\min_i \tau_i \leq 2\delta\}\right] \\
&\leq 1 - e^{-\lambda_K 2\delta}.
\end{align*}
It follows that, as $\delta \to 0$,
\begin{equation*}
\sup_{y \in K} \mathbb{P}_y\left[\{\min_i \tau_i \leq 2\delta\}\right]  \leq 1 - e^{-\lambda_K 2\delta} \to 0,
\end{equation*}
and \ref{item:proof-LDP-zigzag:tightness2} follows from taking $\delta$ small enough that $1 - e^{-\lambda_K 2\delta} < \epsilon$.

We now move to verifying Condition \ref{item:thm_LDP_non_compact:Lyapunov}, by explicitly defining two Lyapunov functions $g_1,g_2 : E \to \mathbb{R}$ satisfying the condition. For brevity, we carry out the calculations for the case of $\gamma (x) \equiv 0$ in the switching rate $\lambda$ (see \eqref{eq:switching-intensity-condition-2}). Then we can use the following functions: for $\alpha_1,\alpha_2 \in (0,1)$ and $\beta > 0$, let
\begin{align*}
g_1(x,v) &= \alpha_1 V(x) + \beta v U'(x),\\
g_2(x,v) &= \alpha_2 U(x) + \beta v U'(x). 
\end{align*}
For non-constant $\gamma$ that is uniformly bounded by some $\bar \gamma$, the following functions can instead be used:
\begin{align*}
g_1(x,v) &= \alpha_1 V(x) + \phi(v U'(x)),\\
g_2(x,v) &= \alpha_2 U(x) + \phi(v U'(x)),
\end{align*}
where $\phi(s) = \beta \frac{1}{2} \text{sign}(s) \log(\bar{\gamma} + |s|)$ and $\beta \in (0,1)$. For example, for the choice $\beta = 1/2$ calculations analogous to the ones below hold.

We now return to the case $\gamma \equiv 0$ and take $g_1, g_2$ accordingly. Without loss of generality we can assume $g_1, g_2 \geq 0$: we can take $\beta$ small enough and if necessary add a constant to ensure that this holds. We show that for suitable $\alpha_i$ small enough, the functions $g_1,g_2$ satisfy \ref{item:thm_LDP_non_compact:Lyapunov}.

By \ref{item:thm-LDP-zig-zag:Upsi}, $V(x) \to \infty$, and by \ref{item:thm-LDP-zig-zag:Upsii}, $U'(x)/V(x) \to 0$. It follows that  $g_1$ grows to infinity as $|x| \to \infty$. Moreover, $g_2$ grows to infinity by the assumption \ref{item:thm-LDP-zig-zag:Ui} on $U$; since $g_1$ and $g_2$ are continuous, this settles part (a) of \ref{item:thm_LDP_non_compact:Lyapunov}. 

Part (b) of \ref{item:thm_LDP_non_compact:Lyapunov} requires that that $g_2$ grows faster than $g_1$ at infinity. This follows from Assumption \ref{item:thm-LDP-zig-zag:Upsii} on the potentials $U$ and $V$: both dominate the derivative $U'$, and $U$ grows faster than $V$.
\smallskip

To show that \ref{item:thm_LDP_non_compact:Lyapunov} holds for the zig-zag process, we show that both $g_1$ and $g_2$ satisfy
\begin{equation*}
e^{-g_i(x,v)} (B e^{g_i})(x,v) \to -\infty \quad |x|\to\infty.
\end{equation*}
Then since $e^{-g_i(x,v)} (B e^{g_i})(x,v)$ is continuous, the compactness of superlevel-sets follows.

By the definition of $g_1, g_2$ and the extended generator $B$, 
%
\begin{align*}
e^{-g_1(x,v)} (B e^{g_1})(x,v) &= \alpha_1 v V'(x) + \beta U''(x) + \max(vU'(x),0) \left[e^{-2v\beta U'(x)} - 1\right],	
\end{align*}
and
\begin{align*}
e^{-g_2(x,v)} (B e^{g_2})(x,v) &= \alpha_2 v U'(x) + \beta U''(x) + \max(vU'(x),0) \left[e^{-2v\beta U'(x)} - 1\right].
\end{align*}
We first verify the condition for $g_2$. For $v = +1$, we have
\begin{equation*}
e^{-g_2(x,+1)} (B e^{g_2})(x,+1) = \alpha_2 U'(x) + \beta U''(x) + \max(U'(x),0) \left[e^{-2 \beta U'(x)} - 1\right].
\end{equation*}
For $x \to + \infty$, we have $U'(x) \to + \infty$ by~\ref{item:thm-LDP-zig-zag:Ui}, so that 
\begin{align*}
e^{-g_2(x,+1)} (B e^{g_2})(x,+1) &= U'(x) \left[\alpha_2 -1 + \beta \frac{U''(x)}{U'(x)} + e^{-2\beta U'(x)}\right] \\
&\sim U'(x) (\alpha_2 - 1) \to -\infty,\quad x \to + \infty,
\end{align*}
since $U''/U' \to 0$ by~\ref{item:thm-LDP-zig-zag:Uiii} and $\alpha_2 < 1$. 
\smallskip

For $x \to -\infty$, we have $U'(x) \to -\infty$ by~\ref{item:thm-LDP-zig-zag:Ui}, in particular $U'(x) < 0$ for large $x$. Hence,
\begin{align*}
e^{-g_2(x,+1)} (B e^{g_2})(x,+1) &= U'(x) \left[\alpha_2 + \beta \frac{U''(x)}{U'(x)}\right] \\
&\sim U'(x) \alpha_2 \to -\infty,\quad x \to - \infty.
\end{align*}
For $v = -1$, the argument is analogous and we omit the details; this concludes the treatment of $g_2$.
\smallskip

We now consider $g_1$. For $v = +1$, 
\begin{align*}
e^{-g_1(x,+1)} (B e^{g_1})(x,+1) &= \alpha_1 V'(x) + \beta U''(x) + \max(U'(x),0) \left[ e^{-2v\beta U'(x)} - 1\right].
\end{align*}
In the limit $x \to + \infty$, $U'(x) \to +\infty$ and $V'(x)/U'(x) \to 0$ by~\ref{item:thm-LDP-zig-zag:Upsii}. It follows that
\begin{align*}
e^{-g_1(x,+1)} (B e^{g_1})(x,+1) &= \alpha_1 V'(x) + \beta U''(x) + U'(x) \left[ e^{-2v\beta U'(x)} - 1\right]\\
&= U'(x) \left[\alpha_1 \frac{V'(x)}{U'(x)} + \beta \frac{U''(x)}{U'(x)} + e^{-2\beta U'(x)} - 1\right]	\\
&\sim  - U'(x)  \to - \infty, \quad x \to + \infty.
\end{align*}
For $x \to -\infty$, similar to the computations for $g_2$,
\begin{align*}
e^{-g_1(x,+1)} (B e^{g_1})(x,+1) &= \alpha_1 V'(x) + \beta U''(x)\\
&= V'(x) \left[\alpha_1 + \beta \frac{U''(x)}{V'(x)}\right] \to - \infty, \quad x \to + \infty,
\end{align*}
since $U''/V' \to 0$ by~\ref{item:thm-LDP-zig-zag:Upsiii} and $V' \to - \infty$ by~\ref{item:thm-LDP-zig-zag:Upsi}. The case $v = -1$ can be handled using similar arguments. 

The preceding computations conclude the verification of Condition~\ref{item:thm_LDP_non_compact:Lyapunov}. We are left with verifying the mixing property \ref{item:thm_LDP_non_compact:mixing}.

Let $\nu_1,\nu_2 \in \mathcal{P}_c(E)$. Then there exists a compact set $K\subseteq E$ with $\nu_1(K) = \nu_2(K) = 1$. To show that \ref{item:thm_LDP_non_compact:mixing} holds, we must find $T,M > 0$ and $\rho_1,\rho_2 \in \mathcal{P}([0,T])$ such that for all $A \in \mathcal{B}(E)$,
\begin{equation*}
\int_0^T\int_E P(t,y,A) \,d\nu_1(y)d\rho_1(t) \leq M \int_0^T\int_E P(t,y,A) \,d\nu_2(y)d\rho_2(t).
\end{equation*}
By Fubini's theorem, it is sufficient to prove that for any points $y_1 \in \text{supp}(\nu_1)$ and $y_2 \in \text{supp}(\nu_2)$,
\begin{equation}\label{eq:proof-LDP-zigzag:sufficient-Fubini}
\int_0^T P(t,y_1,A)\,d\rho_1(t) \leq M \int_0^T P(t,y_2,A)\,d\rho_2(t),
\end{equation}
with $\rho_1,\rho_2,T,M$ independent of $y_1,y_2$. To that end, let $K\subseteq E$ be a compact set containing the support of both $\nu_1$ and $\nu_2$. Without loss of generality, we can take $K$ of the form $K_\mathbb{R} \times \{\pm 1\}$, where $K_\mathbb{R}$ is a closed interval. For $t_1 > 0$, let $K(t_1)$ be the set of points that the zig-zag process with speed one can reach in the time interval $[0,t_1]$ when starting in $K$:
\begin{equation*}
K(t_1) = \left\{y \in E\,:\, \text{dist}_E(y,K) \leq t_1\right\}.
\end{equation*}
We prove the inequality \eqref{eq:proof-LDP-zigzag:sufficient-Fubini} for arbitrary points $y_1,y_2 \in K$, using the following two steps; in what follows we set $\mu = \Leb \otimes \unif_{\pm 1}$.
\begin{enumerate}[label=(\roman*)]
	\item \label{item:proof-zigzag:step_one} For any $t_1 > 0$ and with $\rho_1$ the uniform distribution over $[0,t_1]$, there is a positive constant $C_{K,t_1}$ depending only on $K$ and $t_1$ such that for any $T > t_1$, we have
	\begin{equation*}
	\int_0^T P(t,y_1,A)\,d\rho_1(t) \leq C_{K,t_1} \cdot \mu\left(A \cap K(t_1)\right), \quad \text{ for all } A \in \mathcal{B}(E),
	\end{equation*}
	with $\mu$ as the reference measure on $\mathcal{B}(E)$.
	\item \label{item:proof-zigzag:step_two} There exist positive constants $T > 0$ and $C'_{K,T}$ such that with $\rho_2$ the uniform distribution over $[0,T]$, we have
	\begin{equation*}
	\int_0^T P(t,y_2,A)\,d\rho_2(t) \geq C'_{K,T}\cdot \mu\left(A \cap K(t_1)\right), \quad \text{ for all } A \in \mathcal{B}(E).
	\end{equation*}
\end{enumerate}
Suppose \ref{item:proof-zigzag:step_one} and \ref{item:proof-zigzag:step_two} hold. Then the estimate \eqref{eq:proof-LDP-zigzag:sufficient-Fubini} also holds, with $M = C_{K,t_1}/C'_{K,T}$, some $t_1 < T$.
\smallskip

To verify \ref{item:proof-zigzag:step_one}, note that the measure
\begin{equation*}
\mu_{y_1}(A) = \int_0^T P(t,y_1,A)\,d\rho_1(t)
\end{equation*}
is absolutely continuous with respect to $\mu= \Leb \otimes \unif_{\pm 1}$ and its density is uniformly bounded in $K$. Now~\ref{item:proof-zigzag:step_one} follows since $P(t,K,A) = 0$ whenever $A \cap K(t_1) = \emptyset$ and $t \leq t_1$.
\smallskip

Next, we use Lemma 8 of \cite{bierkens2019ergodicity} to show \ref{item:proof-zigzag:step_two}. To that end, recall that a tuple $(y,y')$ in $E \times E$ is called reachable if there exists an admissible path from $y$ to $y'$. By Theorem 4 of \cite{bierkens2019ergodicity}, any two points are reachable as long as the potential $U$ has at least one non-degenerate local minimum (which is trivially satisfied on $\mathbb R$ under our assumptions) and satisfies $U\in C^3(\mathbb{R})$. 
\smallskip

By Lemma 8 in \cite{bierkens2019ergodicity}, for any two points $y_a = (x_a,v_a)$ and $y_b = (x_b,v_b)$ in $E$, there are open neighborhoods $U_{y_a}$ of $x_a$ and $U_{y_b}$ of $x_b$, a time interval $(t_0,t_0 + \varepsilon]$ and a constant $c > 0$ such that for all $x_a'\in U_{y_a}$ and $t \in (t_0,t_0 + \varepsilon]$,
\begin{equation*}
P\left(t,(x_a',v_a),A_\mathbb{R} \times \{v_b\}\right) \geq c \, \Leb(A \cap U_{z_b}),\quad \text{ for all }A_\mathbb{R} \in \mathcal{B}(\mathbb{R}).
\end{equation*}
The spatial part of $K \times K(t_1)$ can be covered by open squares associated to all pairs of start and final points $y_a$ and $y_b$, with $y_a \in K$ and $ y_b \in K(t_1)$, as
\begin{equation*}
K_\mathbb{R} \times K(t_1)_\mathbb{R} \subseteq \bigcup_{(y_a,y_b)} U_{y_a} \times U_{y_b},
\end{equation*}
where each $U_{y}$ is an open interval in $\mathbb{R}$. By compactness, there exists a finite subcover by open squares $U_{y_a^i} \times U_{y_b^i}$ corresponding to pairs $(y_a^i,y_b^i)$,
\begin{equation*}
K_\mathbb{R} \times K(t_1)_\mathbb{R} \subseteq \bigcup_{i = 1}^N U_{y_a^i} \times U_{y_b^i}.
\end{equation*}
Thereby, the set $K \times K(t_1) \subseteq E \times E$ is covered as
\begin{equation*}
K \times K(t_1) \subseteq \bigcup_{i = 1}^N \left[ \left(U_{y_a^i}\times \{\pm 1\}\right) \times \left(U_{y_b^i}\times \{\pm 1\}\right)\right].
\end{equation*}
Hence for each $z = (x,v) \in K$, there are finitely many open sets $U_{y_b^i}$ covering $K(t_1)_\mathbb{R}$, with corresponding constants $c_i,t_i,\varepsilon_i$ such that for all $t \in (t_i,t_i+\varepsilon_i]$,
\begin{equation}\label{eq:proof-LDP-zigzag:bound-P-below}
P(t,y,A_\mathbb{R} \times \{v_b^i\}) \geq c_i \Leb(A_\mathbb{R} \cap U_{z_b^i}), \quad \text{ for all } A_\mathbb{R} \in \mathcal{B}(\mathbb{R}).
\end{equation}
For any $A = A_\mathbb{R} \times A_\pm \in \mathcal{B}(E)$, write
\begin{align*}
A^+ &:= A \cap (\mathbb{R} \times \{+1\}),\\
A^- &:= A \cap (\mathbb{R} \times \{-1\}).
\end{align*}
Then with $T > 0$ large enough for all intervals $(t_i,t_i+\varepsilon_i]$ to be contained in $[0,T]$, taking $\rho_2 = \unif([0,T])$, for any $z \in K$ it holds that
\begin{align*}
\int_0^T P(t,y,A)\,d\rho_2(t) &\geq \int_0^T P(t,y,A) \sum_i \mathbf{1}_{(t_i,t_i+\varepsilon_i]}(t)\,d\rho_2(t)	\\
&= \frac{1}{T}\sum_i \int_{t_i}^{t_i+\varepsilon_i} \left[P(t,y,A^+) + P(t,y,A^-)\right]\, dt.
\end{align*}
In each time interval $(t_i,t_i+\varepsilon_i]$, at least one transition probability is bounded from below as in~\eqref{eq:proof-LDP-zigzag:bound-P-below}, while the other one can be bounded from below by zero. Thereby,
\begin{align*}
\int_0^T P(t,y,A)\,d\rho_2(t) &\geq \frac{1}{T} \sum_i \varepsilon_i \cdot c_i\cdot \Leb(A_\mathbb{R} \cap U_{y_b^i})	\\
&\geq \frac{1}{T} \min_i(\varepsilon_i c_i) \sum_i \mu\left[A \cap (U_{y_b^i}\times\{\pm 1\})\right]	\\
&\geq \frac{1}{T} \min_i(\varepsilon_ic_i) \cdot \mu \left[A \cap \bigcup_i \left(U_{y_b^i}\times\{\pm 1\}\right) \right]	\\
&\geq \frac{1}{T} \min_i(\varepsilon_ic_i) \cdot \mu\left[A \cap K(t_1)\right],
\end{align*}
where the last inequality follows from $K(t_1)$ being covered by the $U_{y_b^i}\times\{\pm 1\}$. Hence \ref{item:proof-zigzag:step_two} follows with $C_{K,T}' = \min_i(\varepsilon_i c_i)/T$.
\smallskip

This finishes the verification of Condition \ref{item:thm_LDP_non_compact:mixing}, and thereby the proof of Theorem \ref{thm:LDP-zigzag:non-compact-1d}.
\subsection{Derivation of the explicit form of the rate function}
\label{sec:proofs-rate-function}

Here we prove the results described in Section~\ref{sec:rate}. Recall that the state space is now taken as $E = \mathbb{T} \times \{ \pm 1\}$.

Suppose $\mu$ is absolutely continuous with respect to $\nu_0$ and write $\frac{d \mu}{d \nu_0}(x,v) = \rho(x,v)$ for the Radon-Nikodym density of $\mu$ with respect to $\nu_0$, where $\rho$ is assumed to be absolutely continuous.
Define a mapping $H : \mathcal D^+(L) \rightarrow \R$ by
\begin{equation}
\label{eq:H-functional}
 H (u) = \int_E \frac{ L u}{u}  \ d \mu.
\end{equation}
We compute
\begin{equation}
\label{eq:first-steps}
\begin{aligned} H(u) & = \sum_{v\in \{-1,+1\}} \int_{\mathbb{T}} \frac{Lu}{u} (x,v) \rho(x,v) \ d x \\
 & = \int_{\mathbb{T}} \left\{\frac{d \log u^+}{d x} + \lambda^+ \left( \frac{u^-}{u^+} - 1 \right)  \right\} \rho^+ \ d x +\int_{\mathbb{T}} \left\{ - \frac{d \log u^-}{d x} + \lambda^- \left( \frac{u^+}{u^-} - 1 \right) \right\} \rho^- \ d x \\
 & = \int_{\mathbb{T}} \left\{ -\log u^+ \frac{d \rho^+}{dx} + \lambda^+ \rho^+ \left( \frac{u^-}{u^+} - 1 \right) \right\} \ d x +  \int_{\mathbb{T}} \left \{ \log u^- \frac{d \rho^-}{d x} + \lambda^- \rho^-\left( \frac{ u^+}{u^-} - 1 \right) \right\} \ d x.
 \end{aligned}
\end{equation}

\begin{lemma}
\label{lem:to-infinity}
Suppose $\rho \in C(E)$ is absolutely continuous and satisfies
\[ \nu_0  \left\{\dfrac{d\rho^+}{d x} \neq \dfrac{d \rho^-}{dx} \right\} > 0.\]
Then $\inf_{u \in \mathcal D^+(L)} H(u)=-\infty$.
\end{lemma}
\begin{proof}
Let $u^+_t(x) = u^-_t(x) = \exp \left( - t \left\{ \dfrac{d \rho^+}{d x} - \dfrac{d \rho^-}{d x} \right\}\right)$.
From~\eqref{eq:first-steps} it follows that 
\[ H(u_t) = - t \int_{\mathbb{T}} \left( \dfrac{d \rho^+}{d x} - \dfrac{d \rho^-}{dx} \right)^2 \ d x.\] Now let $t \rightarrow \infty$.
\end{proof}

\begin{lemma}
\label{lem:reparametrization}
Suppose $\rho \in C(E)$ is absolutely continuous and $\frac{d \rho^+}{dx} = \frac{d \rho^-}{dx}$ for all $x \in \mathbb{T}$. Then $\mathcal I$ admits the representation
\begin{equation} \label{eq:I-in-eta} \mathcal I(\mu) = - \inf_{\eta \in C(\mathbb{T})} \int_{\mathbb{T}} \left\{ - \rho' \eta + \lambda^+ \rho^+ (\exp(-\eta) - 1) + \lambda^- \rho^- (\exp(\eta) - 1) \right\} \ d x.\end{equation}
\end{lemma}

\begin{proof}
Write $\rho' : =\frac {d \rho^+}{dx}$, and note that by our assumption $\rho' = \frac{d \rho^-}{dx}$.
By~\eqref{eq:first-steps} we may write
\[ H(u) = \int_{\mathbb{T}} \left\{ - \log \left( u^+/u^- \right) \rho' + \lambda^+\rho^+ (u^-/u^+ -1) + \lambda^- \rho^- (u^+/u^- - 1) \right\} \ d x.\]
We see that only the ratio $u^+/u^-$ determines the value of $H(u)$. To any choice of $u \in \mathcal D^+(L)$ we may associate $\eta = \log u^+ - \log u^- \in C^1(\mathbb{T})$, and correspondingly, to any $\eta \in C^1(\mathbb{T})$ we can associate $u \in \mathcal D^+(L)$ by letting
\[ u^+(x) = \exp(\tfrac 1 2 \eta(x)), \quad u^-(x) = \exp(-\tfrac 1 2 \eta(x)), \quad x \in \mathbb{T}.\]
By the continuous dependence of $H$ on $\eta$, and the fact that $C^1(\mathbb{T})$ is dense in $C(\mathbb{T})$, we obtain the stated representation of $I(\mu)$.
\end{proof}

\begin{lemma}
\label{lem:pointwise-minimization}
Suppose $\rho \in C^1(E)$ and $\frac{ d \rho^+}{d x} = \frac{d \rho^-}{dx}$ for all $x \in \mathbb{T}$. Furthermore suppose $\lambda^- \lambda^+ \rho^- \rho^+ > 0$ on $\mathbb{T}$, and $\lambda^{\pm}$ are continuous. Then $\mathcal I$ is given by~\eqref{eq:rate-function}.
\end{lemma}

\begin{proof}
Differentiating the integrand in~\eqref{eq:I-in-eta} pointwise with respect to $\eta$ gives the first order condition
\[ -\rho' - \lambda^+ \rho^+ \exp(-\eta) + \lambda^- \rho^- \exp(\eta) = 0,\]
which is solved uniquely by 
\[ \eta = \tfrac 1 2 \log \left( \frac{\lambda^+ \rho^+}{\lambda^- \rho^-} \right) + \arcsinh \left( \frac{ \rho'}{2 \sqrt{ \lambda^+ \lambda^- \rho^+ \rho^-}}\right),\]
as long as $\lambda^- \lambda^+ \rho^- \rho^+ \neq 0$. Furthermore $\eta \in C(\mathbb{T})$ by the conditions on $\lambda$ and $\rho$.
The second order derivative with respect to $\eta$ is given by
\[ \lambda^+ \rho^+ \exp(-\eta) + \lambda^- \rho^- \exp(\eta) \geq 0,\]
which shows that the critical value of $\eta$ corresponds to a pointwise global minimum of the integrand. 
\end{proof}


\begin{proof}[Proof of Proposition~\ref{prop:expression-ratefunction}]
The expression in case of equality follows from Lemma~\ref{lem:pointwise-minimization}. The result for unequal derivatives follows from Lemma~\ref{lem:to-infinity}.
\end{proof}

\begin{proof}[Proof of Proposition~\ref{prop:no-v-dependence}]
We inspect the dependence of the various terms in the integrand of the expression~\eqref{eq:rate-function}
$I(\mu_c)$ on $c$. For the first term, interchanging integral and derivative,
\begin{align*}
& \frac{d}{dc}  \int_{\mathbb{T}} \rho' \log \left( \frac{\lambda^+ (\rho + c)}{\lambda^- (\rho-c)}\right) \ d x  = \int_{\mathbb{T}} \rho' \left(\frac{1}{\rho + c} + \frac 1 {\rho-c} \right) \ d x \\
 & =  \int_{\mathbb{T}} \frac{d}{dx} \left( \log(\rho +c) + \log(\rho -c) \right) \ d x = 0.
\end{align*}
The following terms (i.e. the $\arcsinh$ and the square root) in the expression for $I(\mu_c)$ are decreasing with respect to the value of $\rho^+ \rho^- = \rho^2 -c^2$. It follows that the integrands are minimized at $c = 0$. Finally, we have that
\[ \int_{\mathbb{T}}  (\lambda^+ \rho^+ + \lambda^- \rho^- )  \ d x = \int_{\mathbb{T}} \left\{  (  \lambda^+ + \lambda^-) \rho + c (\lambda^+ - \lambda^-)  \right\} \ d x.\]
The linear term in $c$ vanishes since $\int_{\mathbb{T}} \{ \lambda^+ - \lambda^- \}  \ d x= \int_{\mathbb{T}} U' \ dx = 0$. It follows that $c = 0$ minimizes $c \mapsto I(\mu_c)$. The stated expression for $I(\mu_0)$ is a straightforward manipulation of~\eqref{eq:rate-function}.
\end{proof}

%
\subsection*{Acknowledgment}
The authors thank Jin Feng for several helpful discussions on topics directly related to this paper. J.\ Bierkens acknowledges support by the Dutch Research Council (NWO) for the research project \textit{Zig-zagging through computational barriers} with project number 016.Vidi.189.043. P.\ Nyquist acknowledges funding from the NWO grant 613.009.101 at the outset of this work. M.\ Schlottke acknowledges financial support through NWO grant 613.001.552.
%
%

\input{main.bbl}

\end{document}

%% file: main.bbl
\newcommand{\etalchar}[1]{$^{#1}$}